\documentclass{article}
\usepackage[utf8]{inputenc}

\usepackage{amsmath}
\usepackage{lipsum}
\usepackage{amsfonts}
\usepackage{graphicx}
\usepackage{epstopdf}
\usepackage{algorithm}
\usepackage{algorithmic}
\usepackage[caption=false]{subfig}
\usepackage{hyperref}
\RequirePackage[amsmath,thmmarks,hyperref]{ntheorem}[1.33]
\theoremstyle{plain}
\newtheorem{lemma}{Lemma}
\newtheorem{proposition}{Proposition}
\newtheorem{definition}{Definition}
\newtheorem{theorem}{Theorem}
\usepackage{enumitem}
\setenumerate[1]{itemsep=1pt,partopsep=0pt,parsep=\parskip,topsep=3pt}
\setitemize[1]{itemsep=1pt,partopsep=0pt,parsep=\parskip,topsep=3pt}

\hypersetup{hidelinks}

\usepackage[verbose=true,letterpaper]{geometry}
\AtBeginDocument{
  \newgeometry{
    textheight=9in,
    textwidth=6.5in,
    top=1in,
    headheight=14pt,
    headsep=25pt,
    footskip=30pt
  }
}
\usepackage{indentfirst}

\newcommand{\newremark}[2]{
  \theoremstyle{plain}
  \theoremheaderfont{\normalfont\itshape}
  \theorembodyfont{\normalfont}
  \theoremseparator{}
  \theoremsymbol{}
  \newtheorem{#1}[theorem]{#2}
}

\newremark{remark}{Remark}
\newremark{assumption}{Assumption}

\newcommand{\newproblem}[2]{
  \theoremstyle{nonumberplain}
  \theoremheaderfont{\normalfont\bfseries}
  \theorembodyfont{\normalfont}
  \theoremseparator{}
  \theoremsymbol{}
  \newtheorem{#1}[theorem]{#2}
}
\newproblem{problem}{Problem}

\theoremstyle{nonumberplain}
\theoremheaderfont{\normalfont\bfseries}
\theorembodyfont{\normalfont}
\theoremseparator{.}
\theoremsymbol{\vbox{\hrule height0.6pt\hbox{\vrule height1.3ex width0.6pt\hskip0.8ex\vrule width0.6pt}\hrule height0.6pt}}
\newtheorem{proof}{Proof}
\newcommand\blfootnote[1]{%
    \begingroup  
    \renewcommand\thefootnote{}\footnote{#1}
    \addtocounter{footnote}{0}
    \endgroup
}

\title{Distributed algorithm for continuous-type Bayesian Nash Equilibrium in Subnetwork Zero-sum Games\blfootnote{This work was supported by the National Natural Science Foundation of China (No. 62173250), and by Shanghai Municipal Science and Technology Major Project (No. 2021SHZDZX0100).}}
\author{Hanzheng Zhang\footnote{Key Laboratory of Systems and Control, Academy of Mathematics and Systems Science, Beijing, China.}\ \footnote{School of Mathematical Sciences, University of Chinese Academy of Sciences, Beijing, China (zhanghanzheng@amss.ac.cn).}, Guanpu Chen\footnote{JD Explore Academy, Beijing, China (chengp@amss.ac.cn).}, Yiguang Hong\footnote{Department of Control Science and Engineering \& Shanghai Research Institute for Intelligent Autonomous Systems, Tongji University, Shanghai, China (yghong@iss.ac.cn).}\ \,\footnotemark[2]}
\date{}

\newcommand\keywords[1]{\textbf{Keywords}: #1}

\begin{document}

\maketitle

\begin{abstract}
In this paper, we consider a continuous-type Bayesian Nash equilibrium (BNE) seeking problem in subnetwork zero-sum games, which is a generalization of deterministic subnetwork zero-sum games and discrete-type Bayesian zero-sum games. In this continuous-type model, because the feasible strategy set is composed of infinite-dimensional functions and is not compact, it is hard to seek a BNE in a non-compact set and convey such complex strategies in network communication. To this end, we design two steps to overcome the above bottleneck. One is a discretization step, where we discretize continuous types and prove that the BNE of the discretized model is an approximate BNE of the continuous model with an explicit error bound. The other one is a communication step, where we adopt a novel compression scheme with a designed sparsification rule and prove that agents can obtain unbiased estimations through compressed communication. Based on the above two steps, we propose a distributed communication-efficient algorithm to practicably seek an approximate BNE, and further provide an explicit error bound and an $O(\ln T/\sqrt{T})$ convergence rate.\\
\par\keywords{Bayesian game, subnetwork zero-sum game, distributed algorithm, equilibrium approximation, communication compression}
\end{abstract}

\section{Introduction}\label{sec_intro}

In recent years, distributed design for decision and control has become more and more important, and distributed algorithms have been proposed for various games \cite{akkarajitsakul2011,basar1989,fang2021,gharesifard2013,xu2022,zeng2011}. With the rapid development of multi-agent systems nowadays, subnetwork zero-sum games, as extensions of zero-sum games, have attracted the attention of researchers. For example, \cite{gharesifard2013} provided a continuous-time distributed equilibrium seeking algorithm for both undirected and directed graphs, while \cite{lou2015} proposed a discrete-time algorithm for equilibrium seeking, and analyzed its convergence. Moreover, \cite{huangs2021} considered an online learning scheme and provided a distributed mirror descent algorithm.

Because of the uncertainties in reality, Bayesian games have attracted a large amount of attention in engineering, computer science, and social science \cite{akkarajitsakul2011,chen2020,grobhans2013,krishna2009}. In Bayesian games, players cannot obtain complete information about the characteristics of the other players, and these characteristics are called types. The distribution of all players' types are public knowledge, while each player knows its own type \cite{harsanyi1967}. Due to the broad applications, the existence and computation of the Bayesian Nash equilibrium (BNE) are thus fundamental problems in the study of various Bayesian games. To this end, many works have investigated BNE with discrete types \cite{akkarajitsakul2011,sola2014,zeng2011} by fixing the types and converting the games to complete-information ones. On this basis, Bayesian zero-sum games, describing a class of zero-sum games with uncertainty, have drawn extensive concerns \cite{bhaskar2016,grobhans2013,li2018}. In addition to the centralized algorithms, there are also many works on distributed Bayesian games \cite{akkarajitsakul2011,sola2014,zeng2011}, where players make decisions based on their own types, local data, and incoming communication through networks.

However, most of the aforementioned works concentrate on discrete-type Bayesian games. In fact, continuous-type Bayesian games are also widespread in engineering and economics \cite{chakraborti2015,krishna2009}. The continuity of types poses challenges in seeking and verifying BNE. Specifically, in continuous-type games, the feasible strategy sets lie in infinite-dimensional spaces which are not compact \cite{milgrom1985}. Due to the lack of the compactness, we cannot apply the fixed point theorem to guarantee the existence of BNE, let alone seek a BNE. To this end, many pioneers have tried to demonstrate the existence of continuous-type BNE and design its computation. For instance, \cite{milgrom1985} analyzed the existence of BNE in virtue of equicontinuous payoffs and absolutely continuous information, while \cite{meirowitz2003} investigated the situation when best responses are equicontinuous. Afterwards, \cite{guos2021} provided an equivalent condition of the equicontinuity and proposed an approximation algorithm for finding a continuous-type BNE. Also, \cite{ui2016} regarded the BNE as the solution to the variational inequality and gave a sufficient condition for the existence of BNE, while \cite{guow2021} gave two variational-inequality-based algorithms when the forms of strategies are prior knowledge.

Considering the development of subnetwork zero-sum games and Bayesian games, it is significant to explore distributed algorithms for seeking BNE in continuous-type Bayesian zero-sum games, since they can be regarded as generalizations of both discrete-type Bayesian zero-sum games \cite{burr2021,li2018} and deterministic subnetwork zero-sum games \cite{gharesifard2013,huangs2021,lou2015}. Nevertheless, the continuous-type models are more challenging to handle than the discrete-type ones in distributed subnetwork zero-sum games. Actually, the challenges come from both continuous types and communication through networks. On the one hand, to seek a continuous-type BNE in a distributed manner, we need an effective method to convert the infinite-dimensional BNE seeking problem into a finite-dimensional one, and the method should be friendly to distributed design. On the other hand, since players need to exchange their strategies with their neighbors, the according strategies, which are infinite-dimensional functions, are hard to be conveyed directly under limited communication capabilities, and we need an effective method to handle the exchange of complex continuous-type strategies under limited communication capabilities.

Therefore, we consider seeking a continuous-type BNE in distributed subnetwork zero-sum games in this paper, where agents in each subnetwork cooperate against the adversarial subnetwork and the two subnetworks are engaged in a zero-sum game. Each subnetwork has its own type following a continuous joint distribution, and each agent knows the type of its own subnetwork. The challenges lie in how to seek a BNE in this continuous-type model and how to efficiently exchange information through the networks. To this end, we introduce a discretization step and a communication compression step to carry forward correspondingly, and propose a distributed BNE seeking algorithm. The contributions are summarized as follows:

(1) We design a distributed algorithm for seeking a BNE in subnetwork zero-sum games. This game model can be regarded as a generalization of discrete-type Bayesian zero-sum games \cite{burr2021,li2018} and deterministic subnetwork zero-sum games \cite{gharesifard2013,huangs2021,lou2015}. With discretizing continuous types and compressing network communication, the algorithm leads to an approximate BNE with an explicit error bound. We show the convergence of the communication-efficient algorithm, as well as its $O(\ln T/\sqrt{T})$ convergence rate.

(2) In the discretization step, to approximate a BNE of continuous-type models, we discretize continuous types in order to make the algorithm implementable. By conducting a distributed-friendly discretization, we prove that the derived BNE sequence in the discretized model converges to the BNE of the continuous model. Moreover, compared with existing works on continuous-type Bayesian games, our method provides an explicit error bound by taking into account of the zero-sum condition \cite{guos2021,huangl2021}, and serves as a practicable method beyond heuristics \cite{grobhans2013,guow2021}.

(3) In the communication step, we adopt compression in the distributed algorithm design to reduce the communication complexity. For this purpose, we design a novel sparsification rule to reduce the communication burden to an acceptable level, since the existing compression methods for optimization \cite{chen2021,wang2018} can hardly be directly applied due to players' interactions here. Correspondingly, we propose a communication scheme to handle the complex interactions of players, which can thus be well adapted to time-varying networks in subnetwork zero-sum games. On this basis, we show that agents can get unbiased estimations of both subnetworks.


The paper is organized as follows. Section \ref{sec_preli} summarizes preliminaries. Section \ref{sec_pf} formulates the continuous-type BNE seeking problem in subnetwork zero-sum games. Section \ref{sec_schedule} outlines a distributed algorithm for seeking BNE. Section \ref{sec_algana} provides the technical details of the algorithm, involving the discretization step and the communication step. Section \ref{sec_alg} gives the detailed algorithm, while Section \ref{sec_conv} provides convergence analysis. Then Section \ref{sec_simu} provides numerical simulations for illustration. Finally, Section \ref{sec_con} concludes the paper.

\section{Preliminaries}\label{sec_preli}

In this section, we give notations and preliminaries about convex analysis, Bayesian games, and graph theory. 

\subsection{Notations}
Denote the $n$-dimensional real Euclidean space by $\mathbb{R}^n$. For $x\in\mathbb{R}$, $\lfloor x\rfloor$ is the greatest integer less than or equal to $x$, and $\lceil x\rceil$ is the least integer greater than or equal to $x$. $\boldsymbol{B}(a,\bar{\varepsilon})$ is a ball with the center $a$ and the radius $\bar{\varepsilon}>0$. Denote $col(x_1,\dots,x_n)=(x_1^T,\dots,x_n^T)^T$ as the column vector stacked with column vectors $x_1,\dots,x_n$ and $I_n\in\mathbb{R}^{n\times n}$ as the identity matrix. For column vectors $x,y\in\mathbb{R}^n$, $\left<x,y\right>$ denotes the inner product, and $\lVert\cdot\rVert$ denotes the 2-norm. For a vector $x\in\mathbb{R}^n$, $[x]_k$ denotes the $k$-th element of $x$ ($k\in\{1,\dots,n\}$). For a matrix $X\in\mathbb{R}^{n\times n}$, $[X]_{ij}$ denotes the element in the $i$-th row and $j$-th column of $X$ ($i,j\in\{1,\dots,n\}$). A function is piecewise continuous if it is continuous except at finite points in its domain. For a differentiable function $f(x_1,\dots,x_n)$, denote $\nabla_i f$ as its gradient with respect to $x_i$.
\subsection{Convex analysis}
A set $C\subseteq\mathbb{R}^n$ is convex if $\lambda z_1+(1-\lambda)z_2\in C$ for any $z_1,z_2\in C$ and $0\leq\lambda\leq1$. A point $z$ is an interior point of $C$ if $\boldsymbol{B}(z,\bar{\varepsilon})\subseteq C$ for some $\bar{\varepsilon}>0$. For a closed convex set $C\in\mathbb{R}^n$, a projection map $\Pi_C:\mathbb{R}^n\to C$ is defined as $\Pi_C(x)=\arg\min_{y\in C}\lVert x-y\rVert$, and holds $\left<x-\Pi_C(x),\Pi_C(x)-y\right>\geq0$ for any $y\in C$. A function $f:\mathbb{R}^n\to\mathbb{R}$ is (strictly) convex if $f(\lambda x_1+(1-\lambda)x_2)(<)\leq \lambda f(x_1)+(1-\lambda)f(x_2)$ for any $x_1,x_2\in\mathbb{R}^n$ and $\lambda\in(0,1)$. For a convex function $f(x)$, $w(x)$ is a subgradient of $f$ at point $x$ if $f(y)\geq f(x)+\left<y-x,w(x)\right>$, $\forall\,y\in \mathbb{R}^n$. The set of all subgradients of convex function $f$ at $x$ is denoted by $\partial f(x)$, which is called the subdifferential of $f$ at $x$. For a convex function $f(x_1,\dots,x_n)$, denote $\partial_{i} f$ as the subdifferential of $f$ with respect to $x_i$. A function $f:\mathbb{R}^n\to\mathbb{R}$ is $\mu$-strongly convex ($\mu>0$) if $f(y)\geq f(x)+w(x)^T(y-x)+\frac{\mu}{2}\lVert y-x\rVert^2,\ \mathrm{for\ any}\ x,y\in\mathbb{R}^n,$
where $w(x)\in\partial f(x)$. Moreover, if $f(x)-\frac{\mu}{2}\lVert x\rVert^2$ is convex, then $f(x)$ is $\mu$-strongly convex.

\subsection{Bayesian game}
Consider a Bayesian game with a set of players $\mathcal{V}=\{1,\dots,n\}$, denoted by $G=(\mathcal{V},\{\mathcal{X}_i\}_{i=1}^n,\Theta,p(\cdot),$
$\{f_i(\cdot)\}_{i=1}^n)$. Player $i\in\mathcal{V}$ has its action set $\mathcal{X}_i\in\mathbb{R}^{m_i}$. The incomplete information of player $i$ is referred to the \textit{type}, that is, a random variable $\theta_i\in\Theta_i\subseteq\mathbb{R}$. Denote $\theta_{-i}\in\Theta_{-i}$ as the type vectors of all players except player $i$. The joint distribution density of types over $\Theta=\Theta_1\times\cdots\times\Theta_n$ is denoted by $p$, with the positive marginal density $p_i(\theta_i)=\int_{\Theta_{-i}}p(\theta_i,\theta_{-i})d\theta_{-i}>0,\ i\in \mathcal{V},\ \theta_i\in\Theta_i.$ Define the conditional probability density $p_i(\theta_{-i}|\theta_i)=p(\theta_i,\theta_{-i})/p_i(\theta_i)$, whose expectation is normalized, \textit{i.e.}, $\int_{\Theta_i} p(\theta_i|\theta_{-i})d\theta_i=1,\ i\in \mathcal{V}.$

Here, each player only knows its own type but not those of its rivals. As in Bayesian games \cite{harsanyi1967}, the joint probability distribution over type sets is public information. The cost function of player $i$ is defined as $f_i:\mathcal{X}_1\times\cdots\times\mathcal{X}_n\times\Theta\to\mathbb{R}$, depending on all players' actions and types. Each player $i$ adopts a strategy $\sigma_i$, which is a measurable function mapping from the type set $\Theta_i$ to the action set $\mathcal{X}_i$. Denote the set of player $i$'s possible strategies by $\Sigma_i$, consisting of all measurable functions mapping from $\Theta_i$ to $\mathcal{X}_i$. Define Hilbert spaces $\mathcal{H}_i$ consisting of functions $\beta:\mathbb{R}\to\mathbb{R}^{m_i}$ with the inner product $\left<\sigma_i,\sigma_i'\right>_{\mathcal{H}_i}=\int_{\theta_i\in\Theta_i} \left<\sigma_i,\sigma_i'\right>p_i(\theta_i)d\theta_i<\infty,\ \sigma_i,\sigma_i'\in\Sigma_i,\ i\in\mathcal{V}.$ Thus, the strategy set $\Sigma_i$ lies in the Hilbert space $\mathcal{H}_i$. If player $i$ adopts the strategy $\sigma_i$, its conditional expectation of the cost according to its conditional probability density is 
$
U_i(\sigma_i,\sigma_{-i},\theta_i)=\int_{\Theta_{-i}}f_i(\sigma_i(\theta_i),\sigma_{-i},\theta_i,\theta_{-i})p_i(\theta_{-i}|\theta_i)d\theta_i,
$
where $\sigma_{-i}$ is the profile of all players' strategies except for player $i$.

\subsection{Graph theory}
A digraph (or directed graph) $\mathcal{G}=(\mathcal{V},\mathcal{E})$ consists of a node set
$\mathcal{V}=\{1,\dots,n\}$ and an edge set $\mathcal{E}=\mathcal{V}\times\mathcal{V}$. Node $j$ is a neighbor of $i$ if $(j,i)\in\mathcal{E}$, and take $(i,i)\in\mathcal{E}$. A path in $\mathcal{G}$ from $i_1$ to $i_k$ is an alternating sequence $i_1e_1i_2\cdots i_{k-1}e_{k-1}i_k$ of nodes such that $e_j=(i_j,i_{j+1})\in\mathcal{E}$ for $j\in\{1,\dots,k-1\}$. The (weighted) associated adjacency matrix $A=([A]_{ij})\in\mathbb{R}^{n\times n}$ is composed of nonnegative adjacency elements $[A]_{ij}$, which is positive if and only if $(j,i)\in\mathcal{E}$. $\mathcal{G}$ is bipartite if $\mathcal{V}$ can be partitioned into two disjoint parts $\mathcal{V}_1$ and $\mathcal{V}_2$ such that $\mathcal{E}\subseteq \bigcup_{l=1}^2 (\mathcal{V}_l\times\mathcal{V}_{3-l})$. Digraph $\mathcal{G}$ is strongly connected if there is a path in $\mathcal{G}$ from $i$ to $j$ for any pair nodes $i,j\in\mathcal{G}$.

Consider a multi-agent network $\Xi$ consisting of two subnetworks $\Xi_1$ and $\Xi_2$, whose agents are $\mathcal{V}_1=\{v_1^1,\dots,v_1^{n_1}\}$ and $\mathcal{V}_2=\{v_2^1,\dots,v_2^{n_2}\}$, respectively.
$\Xi$ is described by a digraph, denoted as $\mathcal{G}=(\mathcal{V},\mathcal{E})$, where $\mathcal{V}=\mathcal{V}_1\cup\mathcal{V}_2$. $\mathcal{G}$ can be partitioned into four digraphs: $\mathcal{G}_l=(\mathcal{V}_l,\mathcal{E}_l)$ with $\mathcal{E}_l=\{(j,i)|\,i,j\in\mathcal{V}_l\}$, $l\in\{1,2\}$, and two bipartite graphs $\mathcal{G}_{l(3-l)}=(\mathcal{V},\mathcal{E}_{l(3-l)})$ with $\mathcal{E}_{l(3-l)}=\{(j,i)|\,j\in\mathcal{V}_l,\,i\in\mathcal{V}_{3-l}\}$, $l\in\{1,2\}$. The set of neighbors in subnetwork $\Xi_k$ ($k\in\{1,2\}$) for agent $v_l^i\in\mathcal{V}$ is denoted by $\mathcal{N}_k^{l,i}$.

\section{Distributed subnetwork zero-sum games}\label{sec_pf}

In this section, we formulate a distributed Bayesian Nash equilibrium seeking problem in subnetwork zero-sum games, and demonstrate the challenges in the study of the proposed problem.

\subsection{Problem formulation}

Consider a Bayesian game $G=(\{\Xi_l\}_{l=1}^2,\{\mathcal{X}_l\}_{l=1}^2,$
$\Theta_1\times\Theta_2,p(\cdot),\{f_l(\cdot)\}_{l=1}^2)$ between subnetworks $\Xi_1$ and $\Xi_2$. For $l\in\{1,2\}$, the action set, type set, joint probability density function, and cost function of $\Xi_l$ are denoted by $\mathcal{X}_l\subseteq\mathbb{R}^{m_l}$, $\Theta_l\subseteq\mathbb{R}$, $p(\theta_1,\theta_2)$, and $f_l(x_1,x_2,\theta_1,\theta_2)$, respectively. Denote the strategy sets by $\Sigma_l\subseteq\mathcal{H}_l$. Each agent $v_l^i$ in $\Xi_l$ has a cost function $f_{l,i}:\mathcal{X}_1\times\mathcal{X}_2\times\Theta_1\times\Theta_2\to\mathbb{R}$, which satisfies $f_l(x_1,x_2,\theta_1,\theta_2)=\frac{1}{n_l}\sum_{i=1}^{n_l}f_{l,i}(x_1,x_2,\theta_1,\theta_2).$ At time $t$, agents exchange information with neighbors in $\Xi_l$ via $\mathcal{G}_l(t)$, and obtain information about $\Xi_{3-l}$ via $\mathcal{G}_{(3-l)l}(t)$. For $l\in\{1,2\}$, each subnetwork aims to choose a strategy $\sigma_l\in\Sigma_l$ to minimize the following global expectation for every $\theta_l\in\Theta_l$.
\begin{equation}\label{eq_expcost}
U_l(\sigma_1,\sigma_2,\theta_l)=\int_{\Theta_{3-l}}f_l(\sigma_1(\theta_1),\sigma_2(\theta_2),\theta_1,\theta_2)p_l(\theta_{3-l}|\theta_l)d\theta_{3-l}. 
\end{equation}
The subnetworks are engaged in a zero-sum game, namely, for any $x_l\in\mathcal{X}_l$ and $\theta_l\in\Theta_l$, $f_1(x_1,x_2,\theta_1,\theta_2)+f_2(x_1,x_2,\theta_1,\theta_2)=0.$

\begin{remark}
Existing works have investigated Bayesian zero-sum games where all players share the same environment state type \cite{bhaskar2016,burr2021,lakshmivarahan1982} called nature-state. In our model, we consider that the cost functions are dependent on both players' implicit types, which was also widely discussed in \cite{harsanyi1967,li2018} and can turn to the nature-state case by taking a distribution with $\theta_1=\theta_2$.
\end{remark}

In Bayesian games, researchers are most concerned about the optimal strategies for players and the equilibria for the whole game. Thus, we introduce the best response strategy and the Bayesian Nash equilibrium\footnote{There are two different definitions of BNE. In \cite{guos2021,meirowitz2003}, a strategy pair is a BNE when each strategy in the pair is a best response at each type, while in \cite{athey2001,milgrom1985}, a BNE satisfies that each strategy in the pair is a best response at almost every type. In this paper, we use the latter definition.}, referring to \cite{harsanyi1967}.
\begin{definition} \label{def_bne} Considering the subnetwork zero-sum Bayesian game $G$,
\begin{enumerate}\renewcommand\labelenumi{{\rm(\alph{enumi})}}
\item for subnetwork $\Xi_l$, a strategy $\sigma_{l*}$ is a best response with respect to the adversarial subnetwork's strategy $\sigma_{3-l}\in\Sigma_{3-l}$ if 
\begin{equation*}
\sigma_{l*}(\theta_l)=\arg\min_{\sigma_l(\theta_l)\in\mathcal{X}_l}U_l(\sigma_1,\sigma_2,\theta_l),\ \mathrm{for\ a.e.\ } \theta_l\in\Theta_l, 
\end{equation*}
where the set of best responses is denoted by $BR_l(\sigma_{3-l})$;
\item a strategy pair $(\sigma_1^*,\sigma_2^*)$ is a Bayesian Nash equilibrium (BNE) of $G$ if 
\begin{equation*}\sigma_l^*\in BR_l(\sigma_{3-l}^*),\ \mathrm{for\ a.e.\ } \theta_l\in\Theta_l. 
\end{equation*}
\end{enumerate}
\end{definition}

In Definition \ref{def_bne}, a best response $\sigma_{l*}$ with respect to a given strategy $\sigma_{3-l}$ means that it is an optimal solution for $\Xi_l$ when $\Xi_{3-l}$ adopts $\sigma_{3-l}$. Moreover, reaching a BNE profile means that there is no strategy that a subnetwork can adopt to yield a lower cost, when the adversarial subnetwork keeps its strategy.

We make the following assumptions for $G$. 
\begin{assumption}\label{ass_game}For $l\in\{1,2\}$,
\begin{enumerate}\renewcommand\labelenumi{(\roman{enumi})}
\item $\mathcal{X}_l$ is nonempty, compact and convex;
\item $\Theta_l$ is compact. Without loss of generality, take $\Theta_l=[\underline{\theta}_l,\overline{\theta}_l]$;
\item $f_{l,i}$ is strictly convex in $x_l\in\mathcal{X}_l$, and $f_l$ is continuously differentiable and $\mu$-strongly convex in $x_l\in\mathcal{X}_l$;
\item for $j\in\{1,2\}$, $f_{l,i}$ is $L_{l,j}$-Lipschitz continuous in $x_j\in\mathcal{X}_j$ for each $x_{3-j}\in\mathcal{X}_{3-j}$, $\theta_1\in\Theta_1$, and $\theta_2\in\Theta_2$, and $f_{l,i}$ is $L_\theta$-Lipschitz continuous in $\theta_j\in\Theta_j$ for each $x_1\in\mathcal{X}_1$, $x_2\in\mathcal{X}_2$, and $\theta_{3-j}\in\Theta_{3-j}$;
\item $p(\theta_1,\theta_2)$ is $L_p$-Lipschitz continuous in $\theta_l\in\Theta_l$ for each $\theta_{3-l}\in\Theta_{3-l}$, and $p_l(\theta_l)$ is positive for every $\theta_l\in\Theta_l$;
\item the digraphs $\mathcal{G}_{l}(t)$ are $\mathcal{R}_0$-jointly strongly connected, \textit{i.e.}, $\bigcup_{r=t}^{t+\mathcal{R}_0} \mathcal{G}_{l}(r)$ is strongly connected for $t\geq0$. Moreover, each agent $v_l^i$ has at least a neighbor in $\Xi_{3-l}$ in a finite time interval, \textit{i.e.}, there exists an integer $\mathcal{S}_0>0$ that agent $v_l^i$'s neighbor set in $\mathcal{G}_{(3-l)l}$ satisfies $\bigcup_{r=t}^{t+\mathcal{S}_0-1}\mathcal{N}_{3-l}^{l,i}(r)\ne\emptyset$ for $t\geq0$.
\end{enumerate}
\end{assumption}

Assumption \ref{ass_game} was widely used in studying Bayesian games and distributed equilibrium seeking problems \cite{guos2021,guow2021,huangl2021,burr2021}. Assumption \ref{ass_game}(i) and (iii) ensure that the expectation $U_{l,i}(\sigma_1,\sigma_2,\theta_l)$ is well defined for each $\theta_l\in\Theta_l$ and $(\sigma_1,\sigma_2)\in\Sigma_1\times\Sigma_2$. Assumption \ref{ass_game}(v) guarantees the atomless property\footnote{A distribution is atomless if the probability of any given value is zero $P(X=x)=0$, and probability density functions of all atomless distributions over the interval $[a,b]$ are Lipschitz continuous.}, which is a common assumption in Bayesian games \cite{guos2021,huangl2021,milgrom1985}. Assumption \ref{ass_game}(vi) holds over a variety of network structures and ensures the connectivity of the network, which was also used in \cite{huangl2021,lou2015}.

In continuous-type Bayesian games, we cannot apply the fixed point theorem to ensure the existence of BNE as in the discrete cases \cite{milgrom1985}. Fortunately, the pioneers have provided the existence condition for the continuous-type BNE using infinite-dimensional variational inequalities \cite{guow2021,ui2016}, which is summarized as follows.
\begin{lemma}\label{lem_exibne}
Under Assumption \ref{ass_game}{\rm(i)-(iii)}, there exists a unique BNE of $G$.
\end{lemma}

With the above guarantee of the existence, our goal is to compute the BNE of the proposed model, and we formulate our problem as follows.

\begin{problem}
Seek the BNE in the continuous-type Bayesian game $G=(\{\Xi_l\}_{l=1}^2,$
$\{\mathcal{X}_l\}_{l=1}^2,\Theta,p(\cdot),\{f_l(\cdot)\}_{l=1}^2)$ via distributed computation through time-varying graphs.
\end{problem}

\subsection{Challenges in seeking BNE}\label{subsec_bne}

In Bayesian games, the continuity of types brings difficulties in computation, since the strategies $\sigma_l$ lie in the infinite-dimensional space $\mathcal{H}_l$. Due to the continuous types in Bayesian games, the strategies are functions defined over the continuous type set $\Theta_l$ rather than finite-dimensional vectors in the discrete cases. As Riesz's Lemma shows \cite{rynne2007}, any infinite-dimensional normed space contains a sequence of unit vectors $\{x_n\}$ with $\lVert x_n-x_m\rVert>\alpha$ for any $0<\alpha<1$ and $n\ne m$. Thus, the closed set $\Sigma_l\subseteq\mathcal{H}_l$ is not compact, which present challenges for seeking BNE. There are only a few attempts to seek a BNE of continuous-type games, and usually these methods are difficult to implement or use heuristic approximations. For example, \cite{guow2021} considered the situation that the forms could be represented by finite parameters, which may not be practical since the forms are usually unavailable, while \cite{guos2021} utilized polynomial approximation to estimate a continuous BNE without an explicit estimation error. Moreover, \cite{huangl2021} adopted a heuristic approximation in a discrete-action Bayesian game, but their algorithm was NP-hard, which is not practical to be implemented in the continuous-action cases.

Thus, since directly seeking a BNE is hard, we introduce the following concept.
\begin{definition}\label{def_epsbne}
For any $\epsilon>0$, a strategy pair $(\widetilde{\sigma}^*_1,\widetilde{\sigma}^*_2)$ is an $\epsilon$-Bayesian Nash equilibrium ($\epsilon$-BNE) of $G$ if for any $\sigma_1\in\Sigma_1$ and $\sigma_2\in\Sigma_2$, 
\begin{equation*}
\begin{aligned}
&\int_{\Theta_1} U_1(\sigma_1,\widetilde{\sigma}^*_2,\theta_1)p_1(\theta_1)d\theta_1\geq\int_{\Theta_1}U_1(\widetilde{\sigma}_1^*,\widetilde{\sigma}^*_2,\theta_1)p_1(\theta_1)d\theta_1-\epsilon,\\
&\int_{\Theta_2} U_2(\widetilde{\sigma}^*_1,\sigma_2,\theta_2)p_2(\theta_2)d\theta_2\geq\int_{\Theta_2} U_2(\widetilde{\sigma}_1^*,\widetilde{\sigma}^*_2,\theta_2)p_2(\theta_2)d\theta_2-\epsilon.
\end{aligned} 
\end{equation*}
\end{definition}

Note that the information is distributed among agents in subnetwork zero-sum games. When seeking a BNE in a distributed manner, agents exchange their strategies with each other through the networks. However, it is difficult to convey complex strategies in the network under the limited communication capabilities. Specifically, if we approximate a BNE with discretization, the dimension of the strategies depends on the number of discrete points. Provided that we select more points to achieve higher accuracy, the dimension of the strategies can be explosively large. It is impractical to convey such complex strategies through networks. Thus, we also need to handle the exchange of information between agents and reduce the communication loads.


\section{Outline of BNE seeking algorithm}\label{sec_schedule}

In this section, we outline a distributed algorithm with two main steps for seeking BNE. Via efficient communication, the algorithm can obtain an approximate BNE with an explicit error bound.

Firstly, we approximate the BNE of the continuous-type model with discretization. For $l\in\{1,2\}$, we select $N_l$ points $\theta_l^1,\dots,\theta_l^{N_l}$ from $\Theta_l$, which satisfy $\underline{\theta}_l=\theta_l^0<\theta_l^1<\dots<\theta_l^{N_l}=\overline{\theta}_l$ and $\lim_{N_1,N_2\to\infty}\theta_l^i-\theta_l^{i-1}=0$ for every $i\in\{1,\dots,N_l\}$. By discretization, the algorithm generates a discretized model, whose strategies can be expressed with finite parameters. On this basis, we use the finite-dimensional strategies in the discretized model to approximate the infinite-dimensional strategies in the continuous model. We will show the detailed implementation of the discretization and its approximation result in Section \ref{sec_dis}.

Secondly, we adopt a compression method to reduce the communication loads. Although the strategies in the discretized model are restricted to finite-dimensional functions rather than the infinite-dimensional ones in the continuous model, discretization still results in high-dimensional strategies, and the network is unable to afford such a large amount of communication. To handle this, we use sparsification operators to reduce the dimension of strategies by compressing the data from $n$-dimensional vectors to $d$-dimensional ones ($d\leq n$), and thus reduce the communication loads. Define the compression ratio as $\rho=d/n$. Specifically, we design a sparsification rule to adapt to time-varying networks, and a new communication scheme to guide agents how to exchange information and estimate states of both subnetworks. The details on communication are presented in Section \ref{sec_compress}.

Thus, we summarize the above procedures and outline the according algorithm in the following, whose detailed version is presented in Section \ref{sec_alg}. 

\begin{algorithm}
\caption{Algorithm for seeking BNE of continuous-type Bayesian games}\label{alg_main}
\begin{algorithmic}
\STATE \textbf{Initialization}: Initialize the variables in the algorithm.
\STATE \textbf{Discretization}: For $l\in\{1,2\}$, select $N_l$ discrete points from $\Theta_l$.
\STATE \textbf{Communication}: For $v_l^i$, send compressed messages to neighbors. Estimate the states of both $\Xi_1$ and $\Xi_2$ respectively.
\STATE \textbf{Update}: Update each agent's strategy based on the estimations of both subnetworks' states.
\end{algorithmic}
\end{algorithm}


Based on Algorithm \ref{alg_main}, we provide the main result of this paper, whose proof is presented in Section \ref{sec_conv}.

\begin{theorem}\label{thm_conv}
Let $(\sigma_1^*,\sigma_2^*)$ be the BNE of game $G$ with continuous types. Under Assumption \ref{ass_game}, Algorithm \ref{alg_main} generates a convergent sequence, and its limit point $(\widetilde{\sigma}_1^*,\widetilde{\sigma}_2^*)$ satisfies that, for $l\in\{1,2\}$, 
\begin{equation*}\lVert \widetilde{\sigma}_l^*-\sigma_l^*\rVert^2_{\mathcal{H}_l}\leq \epsilon, \end{equation*}
where $\epsilon=O(\max_{l\in\{1,2\},i\in\{1,\dots,N_l\}}\{\theta_l^i-\theta_1^{i-1}\})$. Accordingly, as $N_1$ and $N_2$ tend to infinity, then 
\begin{equation*}
\lim_{N_1,N_2\to\infty}\lVert \widetilde{\sigma}_l^*-\sigma_l^*\rVert_{\mathcal{H}_l}=0. 
\end{equation*}
\end{theorem}
Theorem \ref{thm_conv} provides the convergence of Algorithm \ref{alg_main} to an approximate BNE with an explicit error bound between the generated result and the BNE of $G$. 

\begin{remark}
Note that for any atomless distributions, the convergence of a function sequence in $\mathcal{H}_l$ is equivalent to its pointwise convergence for almost every $\theta_l\in\Theta_l$, \textit{i.e.}, $\lim_{n\to\infty}\lVert F_n-F\rVert_{\mathcal{H}_l}= 0$ is equivalent to $\lim_{n\to\infty}\lVert F_n(\theta_l)-F(\theta_l)\rVert=0$ for almost every $\theta_l\in\Theta_l$, where $\{F_n\}$ is a sequence in $\mathcal{H}_l$, and $F\in\mathcal{H}_l$. 
\end{remark}

Theorem \ref{thm_conv} shows that the sequence $\{\widetilde{\sigma}_l^*\}$ converges to the equilibrium $\sigma_l^*$. Different from existing heuristic approximations \cite{guos2021,guow2021}, we provide both the convergence and an explicit error bound by taking full advantage of the zero-sum condition.

In Algorithm \ref{alg_main}, we should set up the number of discrete points $N_l$ and the compression ratio $\rho$ of the sparsification operator according to the actual situation. To achieve high accuracy of the approximation, we need to take more discrete points. However, this brings an excessive communication burden. Although the compression tool can reduce the burden, too small compression ratios will significantly slow down the convergence. Therefore, we need to take appropriate $N_l$ and $\rho$ to reach a trade-off between the accuracy of the approximation and the communication burdens.

\section{Technical details of Algorithm \ref{alg_main}}\label{sec_algana}
Note that there are two main steps in the algorithm, that is, the discretization step and the communication step. In the following, we discuss the two steps in details. 

\subsection{Discretization step}\label{sec_dis}
In this subsection, we provide the discretization step in Algorithm \ref{alg_main}, with its effectiveness in approximating best responses and BNE of the continuous-type model $G$.

Denote the discrete type set of $\Xi_l$ by $\widetilde{\Theta}_l=\{\theta_l^1,\dots,\theta_l^{N_l}\}$. If we regard all types lying in the interval $(\theta_l^{i-1},\theta_l^i]$ as $\theta_l^i$, then the discrete types follow the discrete distribution $$\widetilde{P}(\theta_1^i,\theta_2^j)=\int_{\theta_1^{i-1}}^{\theta_1^i}\int_{\theta_2^{j-1}}^{\theta_2^j}p(\theta_1,\theta_2)d\theta_1d\theta_2,\theta_1^i\in\widetilde{\Theta}_1,\ \theta_2^j\in\widetilde{\Theta}_2.$$ Correspondingly, the marginal distribution $\widetilde{P}_l(\theta_l)=\sum_{\theta_{3-l}\in\widetilde{\Theta}_{3-l}}\widetilde{P}_l(\theta_1,\theta_2)$ and the conditional distribution $\widetilde{P}_l(\theta_{3-l}|\theta_l)=\widetilde{P}(\theta_1,\theta_2)/\widetilde{P}_l(\theta_l)$.

To be specific, here we select the discrete points $\theta_l^i\in\Theta_l$, which satisfy 
\begin{equation}\label{eq_dispoints}
\widetilde{P}_l(\theta_l^i)=1/N_l,\ i=1,\dots,N_l. 
\end{equation}
According to Assumption \ref{ass_game}(v) that $p_l(\theta_l)>0$ for every $\theta_l\in\Theta_l$, as $N_1$ and $N_2$ tends to infinity, $\epsilon$ tends to 0. Since we use $\theta_l^i$ to represent $(\theta_l^{i-1},\theta_l^i]$, we choose the length of such interval as small as possible, which can effectively reduce the error. Moreover, our choice of discrete points is friendly to distributed algorithms. If the marginal distributions are not uniform, agents need to additionally share the marginal distributions with the rivals, and do more computation when updating the strategies.

On this basis, we formulate a discretized model, denoted by $\widetilde{G}=(\{\Xi_l\}_{l=1}^2,$
$\{\mathcal{X}_l\}_{l=1}^2,\widetilde{\Theta}_1\times\widetilde{\Theta}_2,\widetilde{P}(\cdot),\{f_l(\cdot)\}_{l=1}^2)$. In this model, strategies are restricted to $N_l m_l$-dimensional vectors. Denote the strategy sets in $\widetilde{G}$ by $\widetilde{\Sigma}_l$ for $\Xi_l$. For any $\widetilde{\sigma}_1\in\widetilde{\Sigma}_1$ and $\widetilde{\sigma}_2\in\widetilde{\Sigma}_2$, the expectation of the cost function $f_l$ is 
\begin{equation}\label{eq_expcostdis}
\widetilde{U}_l(\widetilde{\sigma}_{1},\widetilde{\sigma}_{2},\theta_l)=\sum_{\theta_{3-l}\in\widetilde{\Theta}_{3-l}}f_l(\widetilde{\sigma}_1(\theta_1),\widetilde{\sigma}_2(\theta_2),\theta_1,\theta_2)\widetilde{P}_l(\theta_{3-l}|\theta_l),\ \theta_l\in\widetilde{\Theta}_l. 
\end{equation}
Correspondingly, denote the expectation of agent $v_l^i$ in $\widetilde{G}$ by $\widetilde{U}_{l,i}(\sigma_1,\sigma_2,\theta_l)$. As Definition \ref{def_bne}, we define the following best response and BNE in $\widetilde{G}$.
\begin{definition}\label{def_disbne}
For the subnetwork $\Xi_l$, a strategy $\widetilde{\sigma}_{l*}$ is a best response with respect to the rivals' strategy $\sigma_{3-l}\in\widetilde{\Sigma}_{3-l}$ in $\widetilde{G}$ if for any $\theta_l\in\widetilde{\Theta}_l$, 
\begin{equation}\label{eq_bne_dis}
\widetilde{\sigma}_{l*}(\theta_l)=\arg\min_{\sigma_l(\theta_l)\in\mathcal{X}_l}\widetilde{U}_l(\sigma_1,\sigma_2,\theta_l). 
\end{equation}
Denote the set of best responses in $\widetilde{G}$ by $BR_l^{N_l}(\sigma_{3-l})$. Moreover, a strategy pair $(\widetilde{\sigma}_1^*,\widetilde{\sigma}_2^*)\in\widetilde{\Sigma}_1\times\widetilde{\Sigma}_2$ is a BNE of the discretized model $\widetilde{G}$, or a DBNE$(N_1,N_2)$ if 
\begin{equation*}\widetilde{\sigma}_l^*\in BR_l^{N_l}(\widetilde{\sigma}^*_{3-l})\ \mathrm{for\ a.e.\ } \theta_l\in\widetilde{\Theta}_l. 
\end{equation*}
\end{definition}

The existence of DBNE can be guaranteed by variational inequalities \cite{guow2021} or Browner fixed point theorem \cite{milgrom1985}. We summarize the result as follows.
\begin{lemma}
Under Assumption \ref{ass_game}{\rm(i)-(iii)}, there exists a unique DBNE$(N_1,N_2)$ of the discretized model $\widetilde{G}$.
\end{lemma}

To approximate the strategies in $G$ with the strategies in $\widetilde{G}$, we extend the domains of the strategies from $\widetilde{\Theta}_l$ to the continuous type set $\Theta_l$. For $l\in\{1,2\}$, define the strategies in $\widetilde{G}$ at any type $\theta\in(\theta_l^{i-1},\theta_l^i]$ as  
\begin{equation}\label{eq_extension}
\widetilde{\sigma}_l(\theta)=\widetilde{\sigma}_l(\theta_l^i). 
\end{equation}
Thus, the strategies in $\widetilde{G}$ can make a response to any type in $\Theta_l$. 

%

Then we estimate the best response of $\Xi_l$ in $G$. In fact, a best response needs to respond to any strategies in $\Sigma_l$, while Definition \ref{def_disbne} only considers the responses to strategies in $\widetilde{\Sigma}_l\subseteq\Sigma_l$. Since we adopt the discrete distribution $\widetilde{P}$ in computing $\widetilde{U}_l$ in Definition \ref{def_disbne} and we cannot compute $\widetilde{U}_l$ for any feasible strategy outside $\widetilde{\Sigma}_l$, we modify the best response as follows.

\begin{definition}\label{def_disbr}
For the subnetwork $\Xi_l$, a strategy $\widetilde{\sigma}_{l*}$ is a best response with respect to the rivals' strategy $\sigma_{3-l}\in\Sigma_{3-l}$ in $\widetilde{G}$ if for any $\theta_l^i\in\widetilde{\Theta}_l$, 
\begin{equation*}
\widetilde{\sigma}_{l*}(\theta_l^i)=\arg\min_{\sigma_l(\theta_l^i)\in\mathcal{X}_l}\int_{\underline{\theta}_{3-l}}^{\overline{\theta}_{3-l}}f_l(\sigma_1(\theta_1),\sigma_2(\theta_2),\theta_1,\theta_2)\big|_{\theta_l=\theta_l^i}\frac{\int_{\theta_l^{i-1}}^{\theta_l^i}p(\theta_1,\theta_2)}{\widetilde{P}_l(\theta_l^i)}d\theta_1d\theta_2. 
\end{equation*}
\end{definition}

The following result shows the relation between the best responses in the discretized model and in the continuous model, whose proof can be found in Apendix \ref{pf_br}.
\begin{lemma}\label{prop_br}
With $\sigma_{3-l}\in\Sigma_{3-l}$, if all the best responses in $BR_l(\sigma_{3-l})$ of $G$ are piecewise continuous, then the best responses in $BR_l^{N_l}(\sigma_{3-l})$ of $\widetilde{G}$ are almost surely the best responses of $G$, as $N_l$ tends to infinity. Specifically, for any $\widetilde{\sigma}^{N_l}_{l*}\in BR_l^{N_l}(\sigma_{3-l})$, there exists $\sigma_{l*}\in BR_l(\sigma_{3-l})$ such that 
\begin{equation*}
\lim_{N_l\to\infty}\widetilde{\sigma}^{N_l}_{l*}(\theta_l)=\sigma_{l*}(\theta_l), \ \mathrm{for\ a.e.\ } \theta_l\in\Theta_l. 
\end{equation*}
\end{lemma}

Proposition \ref{prop_br} shows that each subnetwork can use the discretized model to approximate its best response. On this basis, each subnetwork believes that the best response of $\widetilde{G}$ is a near-optimal strategy of $G$ as the response to any rival's strategy. With such a belief, both subnetworks adopt the best responses of $\widetilde{G}$, and thus the best responses form a DBNE. The next result gives a relation between the derived DBNE of $\widetilde{G}$ and the BNE of $G$, whose proof can be found in Apendix \ref{pf_bne}.

\begin{lemma}\label{thm_bne}
Let $(\widetilde{\sigma}_1^*,\widetilde{\sigma}_2^*)$ be a DBNE$(N_1,N_2)$ obtained from $\widetilde{G}$, and $(\sigma_1^*,\sigma_2^*)$ be the BNE of $G$. Under Assumption \ref{ass_game}{\rm(i)}, {\rm(ii)}, {\rm(iv)}, and {\rm(v)},
\begin{enumerate}\renewcommand\labelenumi{{\rm(\alph{enumi})}}
\item $(\widetilde{\sigma}_1^*,\widetilde{\sigma}_2^*)$ is an $\epsilon_1$-BNE of $G$, where $\epsilon_1=O(\max_{l\in\{1,2\},i\in\{1,\dots,N_l\}}\{\theta_l^i-\theta_1^{i-1}\})$;
\item for $l\in\{1,2\}$, 
\begin{equation*}
\lVert \widetilde{\sigma}_l^*-\sigma_l^*\rVert^2_{\mathcal{H}_l}\leq \frac{4\epsilon_1}{\mu}, 
\end{equation*}
that is to say, as $N_1$ and $N_2$ tend to infinity, $\lim_{N_1,N_2\to\infty} \lVert \widetilde{\sigma}_l^*-\sigma_l^*\rVert_{\mathcal{H}_l}=0.$
\end{enumerate}
\end{lemma}

Lemma \ref{thm_bne} shows the relation between the DBNE$(N_1,N_2)$ of $\widetilde{G}$ and BNE of $G$, and provides an explicit error bound, compared with heuristic approximations \cite{guos2021,guow2021}. With this relation, we can approximate the BNE of $G$ using the DBNE$(N_1,N_2)$ of $\widetilde{G}$.

\begin{remark}We make a few comments on Lemma \ref{thm_bne}.
\begin{enumerate}\renewcommand\labelenumi{(\roman{enumi})}
\item Here we utilize the zero-sum condition to compute the error between the DBNE and the BNE, and obtain the above convergence result. For general cases, Lemma \ref{thm_bne}(a) still holds. However, since the strategy sets $\Sigma_l$ are not compact according to Subsection \ref{subsec_bne}, we cannot obtain Lemma \ref{thm_bne}(b) that the DBNE in general cases converges to the BNE of $G$. 
\item Since the conditions in Lemma \ref{thm_bne} do not require the continuity of types, we can also apply our discretization to games with countable types in order to simplify these infinite-dimensional models and get an approximate BNE.
\end{enumerate} 
\end{remark}

\subsection{Communication step}\label{sec_compress}

In this subsection, we provide the communication step in the algorithm by compressing communication to reduce communication burdens. As discussed in Section \ref{sec_schedule}, the strategy dimension may be explosively large owing to high accuracy. Due to the limited communication capabilities, it is hard to convey such large-dimensional data \cite{alistarh2017,bernstein2018,chen2021,sattler2020,xu2022}. Therefore, our goal in this step is to reduce the communication loads while ensuring that agents can get unbiased estimation of both subnetworks. Here we provide a designed sparsification operator and a novel communication scheme to adapt to time-varying digraphs, and show that agents can get unbiased estimations through our designed efficient communication.

Communication compression is a practical technique for reducing the amount of data. Although communication compression will slow down the convergence and increase the computation, it can still considerably improve the performance by greatly reducing the communication loads. Motivated by \cite{chen2021}, we use sparsification to reduce the size of data in order to meet the limited communication capabilities.

\begin{remark}\label{rmk_compress}
Existing communication compression focused on how to eliminate the bias from compression \cite{alistarh2017,bernstein2018,sattler2020}. In their works, to keep the estimations unbiased, each agent (in a static network \cite{alistarh2017,bernstein2018}) or a center server (in a centralized network \cite{sattler2020}) had to record other agents' states. Besides, such an imperfect communication situation was also investigated in \cite{fang2021,lei2018,yu2021}. However, in a time-varying network, because agents cannot receive neighbors' updates every iteration, the above methods are not practical in our model and we need a novel one to ensure unbiased estimations.
\end{remark}

Following the discretization step, we take a sparsification operator $Q_l:\mathbb{R}^{N_lm_l}\to\mathbb{R}^{N_lm_l}$ for $l\in\{1,2\}$ in our algorithm. Each agent selects and sends $d_l$ out of $N_lm_l$ entries of a $N_lm_l$-dimensional vector to its neighbors, while sending empty messages in the other entries to its neighbors. Denote the compression ratio by $\rho_l=d_l/N_lm_l$. By reducing the dimension of data in communication, the sparsification operator can reduce the communication loads to an affordable level, which is necessary for designing distributed algorithms. Here we provide a simple illustrative example.

\textit{Example 1}. Given $\mathcal{X}_1=\mathcal{X}_2=[0,1]^{10}$. In order to achieve a certain accuracy, we take $N_1=N_2=1000$. Thus, the dimension of agents' strategies are 10000. We adopt two different communication rules to solve the DBNE: (i) take $\rho_l=1$, \textit{i.e.}, without compression; (ii) take $\rho_l=0.1$. Clearly, the amount of communication in (ii) is $1/10$ of (i). Details on the communication data size can be found in Section \ref{sec_simu}.

Since the sparsification operator can destroy the symmetry of the adjacency matrices, following the works on distributed algorithms with asymmetric adjacency matrices \cite{kai2014,chen2021}, we introduce a surplus vector for each agent $v_l^i$, denoted by $ s_{l,i}(t)\in\mathbb{R}^{N_lm_l}$. When exchanging information, agent $v_l^i$ compresses both state vector (strategy) and surplus vector and sends the compressed message to current neighbors in $\Xi_l$, and then send compressed states of $\Xi_l$ to current neighbors in $\Xi_{3-l}$.

Because the incoming messages are high-dimensional vectors and hard to analyze, we split the vector-valued communication problem each entry by entry into individual scale-valued sub-problems. In each sub-problem, agents regard those who send nonempty messages as their neighbors. Denote the graph sequence for the $k$-th entry by $\mathcal{G}_{l,k}(t)$, $l\in\{1,2,12,21\}$. For $l\in\{1,2\}$, define the in-neighbor adjacency matrix $A_{l,k}(t)$ of $\mathcal{G}_{l,k}(t)$ for the $k$-th entry ($k\in\{1,\dots,2N_lm_l\}$) of the sparsified vectors $col(Q_l(\sigma_{l,i}(t)),Q_l( s_{l,i}(t))$ at time $t$ as 
\begin{equation}\label{eq_def_A}
[A_{l,k}(t)]_{ij}=\left\{\begin{aligned}&\frac{1}{|\mathcal{A}_{l,i,k}(t)|}\quad\mathrm{if}\ v^j_l\in\mathcal{A}_{l,i,k}(t),\\&0\qquad\qquad\ \ \mathrm{otherwise},\end{aligned}\right. 
\end{equation}
where $\mathcal{A}_{l,i,k}(t)=\{v_l^i\}\cup\{v_l^j|v_l^j\in\mathcal{N}^i_l(t),[Q_l(\sigma_{l,j}(t))]_k\ne\emptyset\}$ is the in-neighbor set of $v_l^i$ in $\mathcal{G}_{l,k}(t)$. Similarly, for $l\in\{1,2\}$, define the out-neighbor adjacency matrix $B_{l,k}(t)$ of $\mathcal{G}_{l,k}(t)$ for the $k$-th entry of $col(Q_l(\sigma_{l,i}(t)),Q_l( s_{l,i}(t))$ ($k\in\{1,\dots,2N_lm_l\}$) as 
\begin{equation}\label{eq_def_B}
[B_{l,k}(t)]_{ij}=\left\{\begin{aligned}&\frac{1}{|\mathcal{B}_{l,i,k}(t)|}\quad \mathrm{if}\ v_l^j\in\mathcal{B}_{l,i,k}(t),\\&0\qquad\qquad\ \ \mathrm{otherwise},\end{aligned}\right. 
\end{equation}
where $\mathcal{B}_{l,i,k}(t)=\{v_l^i\}\cup\{v_l^j|v_l^i\in\mathcal{N}^j_l(t),[Q_l(\sigma_{l,j}(t))]_k\ne\emptyset\}$ is the out-neighbor set of $v_l^i$ in $\mathcal{G}_{l,k}(t)$. 

Then each agent $v_l^i$ compresses the estimation $\hat{\sigma}_{l,i}(t)$ of $\Xi_l$ and send it to neighbors in $\Xi_{3-l}$. For $l\in\{1,2\}$, define the in-neighbor adjacency matrix $C_{l,k}(t)$ of $\mathcal{G}_{(3-l)l,k}(t)$ for the $k$-th entry ($k\in\{1,\dots,N_{3-l}m_{3-l}\}$) of $Q_{3-l}(\hat{\sigma}_{3-l,i}(t))$ at time $t$ as 
\begin{equation}\label{eq_def_C}
[C_{l,k}(t)]_{ij}=\left\{\begin{aligned}&\frac{1}{|\mathcal{C}_{l,i,k}(t)|}\quad \mathrm{if}\ v_{3-l}^j\in\mathcal{C}_{l,i,k}(t),\\&0\qquad\qquad\ \mathrm{otherwise},\end{aligned}\right. 
\end{equation}
where $\mathcal{C}_{l,i,k}(t)=\{v_{3-l}^j|v_{3-l}^j\in\mathcal{N}^i_{(3-l)l}(t),[Q_{3-l}(\hat{\sigma}_{3-l,i}(t))]_k\ne\emptyset\}$ is the in-neighbor set of $v_l^i$ in $\mathcal{G}_{(3-l)l,k}(t)$.

Considering the difficulties of communication in time-varying networks, we provide an effective sparsification rule to satisfy the connectivity of $\mathcal{G}_{l,k}(t)$. For $q\geq 0$ and $x(t)\in\mathbb{R}^{N_lm_l}$, when $t=q\mathcal{R}_0+1,\dots,(q+1)\mathcal{R}_0$, take the sparsification $Q_l$ as  
\begin{equation*}
[Q_l(x(t))]_k=[x(t)]_k, k\in\{ (qd_l+1)\bmod N_lm_l,\dots,(q+1)d_l\bmod N_lm_l\}, 
\end{equation*} while $[Q_l(x(t))]_k$ is empty for other $k$. Define $\mathcal{R}=\max_l\{\mathcal{R}_0\lceil N_lm_l/d_l\rceil\}$ and $\mathcal{S}=\max_l\{\mathcal{S}_0\mathcal{R}_0\lceil N_{l}m_{l}/d_{l}\rceil\}$. Under the above sparsification, we have the following proposition for the connectivity of $\mathcal{G}_{l,k}(t)$, whose proof is presented in Apendix \ref{pf_com_graph}. 

\begin{proposition}\label{prop_compress_graph}
Under Assumption \ref{ass_game}{\rm(vi)}, for $l\in\{1,2\}$,
\begin{enumerate}\renewcommand\labelenumi{{\rm(\alph{enumi})}}
\item there exists an integer $\mathcal{R}>0$ that for each $k\in\{1,\dots,2N_lm_l\}$, the digraph $\mathcal{G}_{l,k}(t)$ is $\mathcal{R}$-jointly strongly connected, \textit{i.e.}, $\bigcup_{r=t}^{t+R-1} \mathcal{G}_{l,k}(r)$ is strongly connected for any $t\geq0$;
\item there exists an integer $\mathcal{S}>0$ that for each $k\in\{1,\dots,N_{3-l}m_{3-l}\}$, each agent's neighbor set in $\mathcal{G}_{(3-l),l}(t)$ satisfies $\bigcup_{r=t}^{t+\mathcal{S}-1}\mathcal{C}_{l,i,k}(r)\ne\emptyset$ for any $t\geq0$.
\end{enumerate} 
\end{proposition}

Different from general sparsification operators \cite{chen2021} that randomly select entries, our sparsification is deterministic and can better adapt to time-varying networks in our model. Note that our sparsification design is to satisfy the worst case. For specific networks, we can improve the sparsification to get smaller $\mathcal{R}$ and $\mathcal{S}$. Based on the above sparsification, we design the following communication scheme to show how agents communicate and make estimations of both subnetworks.

\textbf{Communication scheme.}\label{scheme}
\begin{enumerate}
\item Send compressed state vector $Q_l(\sigma_{l,i}(t))$ to neighbors in $\Xi_l$.
\item Estimate the $k$-th entry ($k\in\{1,\dots,N_lm_l\}$) of $\Xi_{l}$'s states based on 
\begin{equation}\label{eq_estinsub}
[\hat{\sigma}_{l,i}(t)]_k=\sum_{j=1}^{n_l}[A_{l,k}(t)]_{ij}[\sigma_{l,j}(t)]_k. 
\end{equation} 
\item Send compressed estimation $Q_l(\hat{\sigma}_{l,i}(t))$ to neighbors in $\Xi_{3-l}$. 
\item Estimate the $k$-th entry ($k\in\{1,\dots,N_lm_l\}$) of $\Xi_{3-l}$'s states based on  \begin{equation}\label{eq_estintersub}
[\hat{\zeta}_{l,i}(t)]_k=\left\{\begin{aligned}
&\sum_{j=1}^{n_{3-l}} [C_{l,k} (t)]_{ij} [\hat{\sigma}_{3-l,j}(t)]_k\ \mathrm{if}\ |\mathcal{C}_{l,i,k}(t)|\ne 0,\\
&[\hat{\zeta}_{l,i}(t-1)]_k\qquad\qquad\qquad \mathrm{otherwise}.
\end{aligned}\right. 
\end{equation}
\end{enumerate}

Since agents might receive empty messages at some entries in inter-subnetwork communication, we use historical information to supplement the estimations in \eqref{eq_estintersub}. Define the average state $\bar{\sigma}_l(t)=\frac{1}{n_l}\sum_{i=1}^{n_l}(\sigma_{l,i}(t)+ s_{l,i}(t))$. The following result reveals that the estimations are unbiased, whose proof is presented in Apendix \ref{pf_com}. 

\begin{lemma}\label{prop_compress}
Let Assumption \ref{ass_game}(vi) hold. For given $\varepsilon_0,\varepsilon_1>0$, suppose that there exists an integer $T>0$, such that for any $t>T$, $\lVert\sigma_{l,i}(t)-\bar{\sigma}_l(t)\rVert\leq\varepsilon_0$ and $\lVert\bar{\sigma}_l(t+1)-\bar{\sigma}_l(t)\rVert\leq\varepsilon_1$. On this basis, for $l\in\{1,2\}$, 
\begin{enumerate}\renewcommand\labelenumi{{\rm(\alph{enumi})}}
\item $\lVert\hat{\sigma}_{l,i}(t)-\bar{\sigma}_l(t)\rVert\leq\sqrt{N_lm_l}\varepsilon_0$, for any $t>T$;
\item $\lVert\hat{\zeta}_{3-l,i}(t)-\bar{\sigma}_{l}(t)\rVert\leq N_lm_l\varepsilon_0+\mathcal{S}\sqrt{N_lm_l}\varepsilon_1$ for any $t>T+\mathcal{S}$.
\end{enumerate} 
\end{lemma}

Note that the conditions in Proposition \ref{prop_compress} will be guaranteed by the convergence of Algorithm \ref{alg_main}. Proposition \ref{prop_compress} provides explicit error bounds between the estimations and the average states, which means that the designed sparsification rule and the communication scheme are effective in exchanging information. 

\section{Distributed algorithm}\label{sec_alg}

Due to the special update rule of surplus-based algorithms, we cannot employ widely-used constraint methods, such as projection and Lagrange multiplier. Denote strategy $\sigma_l\in\widetilde{\Sigma}_l$ at type $\theta_l^r\in\widetilde{\Theta}_l$ by $\sigma_l^r$. Thus, we give the following penalty function $H_l(x):\mathbb{R}^{m_l}\to\mathbb{R}$ to ensure the generated results lie in the action set $\mathcal{X}_l$. 
\begin{equation}
H_l(\sigma_{l,i}^r(t))=E_l\lVert \sigma_{l,i}^r(t)-\Pi_{X_l}(\sigma_{l,i}^r(t))\rVert, 
\end{equation}
where $E_l>L_{l,l}$ is a constant. 
\begin{proposition}\label{prop_penalty}
For $l\in\{1,2\}$, $\sigma_{3-l}\in\widetilde{\Sigma}_{3-l}$, and $r\in\{1,\dots,N_l\}$,
\begin{enumerate}\renewcommand\labelenumi{{\rm(\alph{enumi})}}
\item $H_l(x)$ is convex and $E_l$-Lipschitz continuous in $\mathbb{R}^{m_l}$;
\item $h_l(x)=E_l \frac{x-\Pi_{\mathcal{X}_l}(x)}{\lVert x-\Pi_{\mathcal{X}_l}(x)\rVert}$ is a subgradient of $H_l(x)$;
\item all the minimizers $\sigma_l^r$ of $\widetilde{U}_l(\sigma_1,\sigma_2,\theta_l^r)+H_l(\sigma_l^r)$ are in the action set $\mathcal{X}_l$.
\end{enumerate} 
\end{proposition}
The proof of Proposition \ref{prop_penalty} is presented in Apendix \ref{pf_penalty}. Since $H_l(x)=0$ for $x\in\mathcal{X}_l$, by Proposition \ref{prop_penalty}(c), all the equilibria of game $\widetilde{G}$ with the expectation of cost $\widetilde{U}_l+H_l$ are DBNE. Thus, we take $\widetilde{U}_l+H_l$ instead of $\widetilde{U}_l$ and ignore the constrains to seek a DBNE. 

With the estimation $\hat{\zeta}_{l,i}(t)$, agent $v_l^i$ evaluates its subgradient by 
\begin{equation}\label{eq_subgrad}
g_{l,i}(t)=col(g_{l,i}^1(t),\dots,g_{l,i}^{N_l}(t)), \end{equation}
where $g_{l,i}^r(t)=(w_{l,i}^r(t)+h_{l,i}^r(t))/N_l$ and $w_{l,i}^r(t)$ is a subgradient of $\widetilde{U}_{l,i}$ according to 
\begin{equation*}
w_{1,i}^r(t)\in\partial_1  \widetilde{U}_{1,i}(\sigma_{1,i}(t),\hat{\zeta}_{1,i}(t),\theta_1^r),\quad
w_{2,i}^r(t)\in\partial_2  \widetilde{U}_{2,i}(\hat{\zeta}_{2,i}(t),\sigma_{2,i}(t),\theta_2^r), 
\end{equation*}
and $h_{l,i}^r(t)= h_l(\sigma_{l,i}^r(t))$, $r\in\{1,\dots,N_l\}$. Define $$\bar{M}_{l,k}(t)=\begin{bmatrix}A_{l,k}(t)&0\\I-A_{l,k}(t)&B_{l,k}(t)\end{bmatrix},\ F=\begin{bmatrix}0&I\\0&-I\end{bmatrix}.$$ Denote the state of subnetwork $\Xi_l$ in the $k$-th sub-problem $(k\in\{1,\dots,N_lm_l\})$ by $z_{l,k}(t)=col([\sigma_{l,1}(t)]_k,\dots,$
$[\sigma_{l,n_l}(t)]_k,[s_{l,1}(t)]_k,\dots,[ s_{l,n_l}(t)]_k)$. Here we give a detailed version of Algorithm \ref{alg_main}. 
\begin{algorithm}[!h]
\renewcommand{\thealgorithm}{4.1}
\caption{(a detailed version)}
\begin{algorithmic}
\STATE \textbf{Initialization}: For $l\in\{1,2\}$: let $\sigma_{l,i}(0)= s_{l,i}(0)=\hat{\zeta}_{3-l,j}(0)$$=x_{l0}\in\widetilde{\Sigma}_l$ for each $i\in\{1,\dots,n_l\}$ and $j\in\{1,\dots,n_{3-l}\}$. 	
\STATE \textbf{Discretization}: For $l\in\{1,2\}$, take $N_l$ points from the type set $\Theta_l$ as \eqref{eq_dispoints}.
\STATE Iterate until $t\geq T$:
\STATE \textbf{Communication}: Agent $v_l^i\in\mathcal{V}_l$ communicates and makes estimations $\hat{\sigma}_{l,i}(t)$ and $\hat{\zeta}_{l,i}(t)$ of both subnetworks respectively based on \hyperref[scheme]{communication scheme}.
\STATE \textbf{Update}: Agent $v_l^i$ evaluate the subgradients $g_{l,i}(t)$ based on \eqref{eq_subgrad}, and updates $\sigma_{l,i}(t)$ and $ s_{l,i}(t)$ for each $k\in\{1,\dots,N_lm_l\}$ by  
\begin{align}
\notag[z_{l,k}(t+1)]_i=&\sum_{j=1}^{2n_l}\big([\bar{M}_{l,k}(t)]_{ij}[z_{l,k}(t)]_j+1_{\{t\bmod \mathcal{R}=\mathcal{R}-1\}}\eta[F]_{ij} [z_{l,k}(\mathcal{R}\lfloor t/\mathcal{R}\rfloor)]_j\big)\\\label{eq_update_rule1}&-1_{\{t\bmod \mathcal{R}=\mathcal{R}-1\}}\alpha(\lfloor t/\mathcal{R}\rfloor)[g_{l,i}(\mathcal{R}\lfloor t/\mathcal{R}\rfloor)]_k,
\end{align}
\end{algorithmic}
\end{algorithm}

Define $M_{l,k}(q+1:q)=\bar{M}_{l,k}(q\mathcal{R})\cdots\bar{M}_{l,k}(q(\mathcal{R}+1)-1)+\eta F$, and $M_{l,k}(q_2:q_1)=M_{l,k}(q_1+1:q_1)\cdots M_{l,k}(q_2:q_2-1)$ ($q_1<q_2$). The parameters $\alpha(t)$ and $\eta$ in Algorithm \ref{alg_main} satisfy the following properties.
\begin{enumerate}\renewcommand\labelenumi{(\alph{enumi})}
\item $\{\alpha(t)\}_{t=0}^{\infty}$ is a positive non-increasing sequence satisfying $\sum_{t=0}^\infty \alpha(t)=\infty$ and $\sum_{t=0}^{\infty}\alpha^2(t)<\infty$;
\item $\eta\in(0,\Upsilon)$, where $\Upsilon=\min_l\{1/(20+8n_l)^{n_l}(1-|\lambda_{3}|)^{n_l}\}$, and $\lambda_3$ is the third largest eigenvalue of $M_{l,k}(q+1:q)$ when $\eta=0$.
\end{enumerate}
The above parameter settings are also similarly investigated in distributed subgradient algorithms \cite{nedic2015,rams2009} and distributed surplus-based algorithms \cite{kai2014,chen2021}.

Note that the stepsize we adopt in Algorithm \ref{alg_main} is $o(1/\sqrt{t})$. Following the subgradient method \cite{nedic2015}, we adjust the stepsize to $\alpha(t)=1/\sqrt{t}$ in the following theorem to obtain the convergence rate. Since the stepsize which decays faster usually brings slower convergence rate, the adjustment can help us analyze the convergence rate, even though the algorithm at this stepsize does not converge.
\begin{theorem}[Convergence rate]\label{thm_conv_rate}Take the stepsize $\alpha(t)=1/\sqrt{t}$ for $t\geq0$. Under Assumption \ref{ass_game}, Algorithm \ref{alg_main} attains a convergence rate $O(\ln T/\sqrt{T})$.
\end{theorem}

Theorem \ref{thm_conv_rate} shows that the convergence rate of Algorithm \ref{alg_main} is consistent with existing distributed algorithms \cite{chen2021,nedic2015}, that is to say, the compression does not affect the order of magnitude of the convergence rate.

\section{Convergence analysis}\label{sec_conv}
In this section, we prove the convergence of Algorithm \ref{alg_main} and provide its convergence rate, namely we present the detailed proof of Theorem \ref{thm_conv} and Theorem \ref{thm_conv_rate}. We first provide the following lemmas on surplus-based algorithms and convergence analysis.
\begin{lemma}\label{lem_consensus}{\rm\cite{kai2014}}.
For $k=1,\dots,N_lm_l$, there exists $\Gamma_l=\sqrt{2n_l m_l}>0$ such that 
\begin{equation*}\lVert M_{l,k}(q:0)-\frac{1}{n_l}[1^T\ 0^T]^T[1^T\ 1^T]\rVert_\infty\leq \Gamma_l\xi_l^q, \end{equation*}
where $\xi_l=\max |\lambda_2(M_{l,k}(q+1:q))|$ and $\lambda_2$ is the second largest eigenvalue of $M_{l,k}(q+1:q)$.
\end{lemma}

\begin{lemma}\label{lem_conv}{\rm\cite{Polyak1987}}. Let $\{a_t\}$, $\{b_t\}$, and $\{c_t\}$ be non-negative sequences satisfying $\sum_{t=0}^\infty b_t<\infty$. If $a_{t+1}\leq a_t+b_t-c_t$ for any $t$, then $a_t$ converges to a finite number and $\sum_{t=0}^T c_t<\infty$.
\end{lemma}

With the above two lemmas and the results developed in Section \ref{sec_algana}, we can carry forward the proof of the main results in the following.
\begin{proof}[Proof of Theorem \ref{thm_conv}]
Recalling Lemma \ref{thm_bne} that the DBNE of $\widetilde{G}$ is an approximate BNE of the continuous model with an explicit error bound, we only need to prove that Algorithm \ref{alg_main} generates a sequence that converges to the DBNE of $\widetilde{G}$. To this end, we show that agents reach consensus through compressed communication and the average state $\bar{\sigma}_l(t)$ converges to the equilibrium strategy $\widetilde{\sigma}^*$ of $\widetilde{G}$.

Firstly, we show that agents can reach consensus, \textit{i.e.}, agents' states $\sigma_{l,i}(t)$ converge to the average state $\bar{\sigma}_l(t)$ and the surplus vectors $s_{l,i}(t)$ vanish. From the update rule \eqref{eq_update_rule1}, for each $l\in\{1,2\}$, $i\in\{1,\dots,2n_l\}$, $k\in\{1,\dots,N_lm_l\}$, and $q\geq1$, the vector $z_{l,k}(q\mathcal{R})$ can be reconstructed as 
\begin{equation*}
\begin{aligned}
[z_{l,k}(q\mathcal{R})]_i&=\sum_{j=1}^{2n_l}[M_{l,k}(q:0)]_{ij}[z_{l,k}(0)]_j\\\notag-&\sum_{r=0}^{q-2}\sum_{j=1}^{n_l}[M_{l,k}(k-1:r)]_{ij}\alpha(r)[g_{l,j}(r\mathcal{R})]_k-\alpha(q-1)[g_{l,i}((q-1)\mathcal{R})]_k.
\end{aligned} 
\end{equation*}
Because the column sum in $M_{l,k}(q_2:q_1)$ equal to 1 for any $q_2>q_1$, the average state $\bar{\sigma}_l(q)$ can be represented as 
\begin{equation*}
[\bar{\sigma}_l(q\mathcal{R})]_k=\frac{1}{n_l}\sum_{j=1}^{2n_l}\bigg([z_{l,k}(0)]_j-\sum_{r=0}^{q-2}\alpha(r)[g_{l,j}(r\mathcal{R})]_k\bigg)-\frac{1}{n_l}\sum_{j=1}^{n_l}\alpha(q-1)[g_{l,j}((q-1)\mathcal{R})]_k.
\end{equation*}
Note that $[\sigma_{l,i}(t)]_k=[z_{l,k}(t)]_i$ and $[s_{l,i}(t)]_k=[z_{l,k}(t)]_{n_l+i}$. According to Lemma \ref{lem_consensus}, we combine the above two equations as 
\begin{equation*}
\begin{aligned}
\lVert [\bar{\sigma}_l(q\mathcal{R})]_k-[\sigma_{l,i}(q\mathcal{R})]_k\rVert
\leq&\Gamma_l\xi_l^{q}\sum_{i=1}^{2n_l}\lVert [z_{l,k}(0)]_i\rVert+\Gamma_l \sum_{r=0}^{q-2}\xi_l^{q-r-1}\alpha(r)\sum_{j=1}^{n_l}\lVert [g_{l,j}(r\mathcal{R})]_k\rVert
\\&+\alpha(q-1)\lVert[g_{l,i}((q-1)\mathcal{R})]_k-\frac{1}{n_l}\sum_{j=1}^{n_l}[g_{l,j}((q-1)\mathcal{R})]_k\rVert.
\end{aligned} \end{equation*}
Define $D_l=L_{l,l}+E_l$. The subgradient satisfies $\lVert g_{l,i}(t)\rVert \leq (L_{l,l}+E_l)/\sqrt{N_l}=D_l/\sqrt{N_l}$. Since $\sum_{k=1}^{n}\lVert [x]_k\rVert^2=\lVert x\rVert^2$ for $x\in\mathbb{R}^{n}$ and $\sqrt{n}\lVert x\rVert\geq\sum_{k=1}^{n}\lVert [x]_k\rVert \geq \lVert x\rVert$, we obtain 
\begin{equation*}
\begin{aligned}
\lVert \bar{\sigma}_l(q\mathcal{R})-\sigma_{l,i}(q\mathcal{R})\rVert&\leq\sum_{k=1}^{N_lm_l}\lVert [\bar{\sigma}_l(q\mathcal{R})]_k-[\sigma_{l,i}(q\mathcal{R})]_k\rVert\\
&\leq \Gamma_lR_l\xi_l^{q}+D_{l}\sqrt{m_l}\left ( n_l\Gamma_l \sum_{r=0}^{q-2}\xi_l^{q-r-1}\alpha(r)+2\alpha(q-1)\right),
\end{aligned} 
\end{equation*}
where $R_l=\sum_{i=1}^{2n_l}\sum_{k=1}^{N_lm_l}\lVert [z_{l,k}(0)]_i\rVert$. Similarly, 
\begin{equation*}
\lVert  s_{l,i}(q\mathcal{R}) \rVert
\leq\Gamma_lR_l\xi_l^{q}+ \Gamma_l D_{l}\sqrt{m_l}n_l\sum_{r=0}^{q-2}\xi_l^{q-r-1}\alpha(r).
\end{equation*}
Furthermore, denote $\alpha(q)=0$ and $\xi^q=0$ for all $q<0$. Then 
\begin{equation*}
\begin{aligned}
&\sum_{q=0}^{\infty}\alpha(q)\lVert \bar{\sigma}_l(q\mathcal{R})-\sigma_{l,i}(q\mathcal{R})\rVert\\
\leq &\sum_{q=0}^{\infty}\bigg(\Gamma_lR_l\xi_l^{q}\alpha(q)+D_{l}\sqrt{m_l}\bigg(n_l\Gamma_l\sum_{r=0}^{q-2}\xi_l^{q-r-1}\alpha(r)\alpha(q)+2\alpha(q-1)\alpha(q)\bigg)\bigg).
\end{aligned} 
\end{equation*}
Applying the inequality $ab\leq (a+b)^2/2$, for any $a,b\in\mathbb{R}$, 
\begin{equation*}\sum_{r=0}^{q-2}\xi_l^{q-r-1}\alpha(r)\alpha(q)\leq \sum_{r=0}^{q-2}\frac{1}{2}\xi_l^{q-r-1}(\alpha(r)^2+\alpha(q)^2). \end{equation*}
Then 
\begin{align}\label{eq_con_rate}
&\sum_{q=0}^{\infty}\alpha(q)\lVert \bar{\sigma}_l(q\mathcal{R})-\sigma_{l,i}(q\mathcal{R})\rVert\\\notag\leq &\sum_{q=0}^{\infty}\bigg(\Gamma_lR_l\xi_l^{q}\alpha(q)+D_{l}\sqrt{m_l}\bigg(n_l\Gamma_l\sum_{r=0}^{q-2}\xi_l^{q-r-1}\alpha^2(r)+\alpha^2(q)+2\alpha^2(q-1)\bigg)\bigg)
\\\notag\leq&\frac{\Gamma_lR_l\alpha(0)}{1-\xi_l}+D_{l}\sqrt{m_l}\left( \frac{n_l\Gamma_l\xi_l}{1-\xi_l}+2\right)\sum_{q=0}^\infty \alpha^2(q)<\infty.
\end{align}
Similarly, 
\begin{equation*}
\begin{aligned}
\sum_{q=0}^{\infty}\alpha(q)\lVert s_{l,i}(q\mathcal{R})\rVert&\leq \sum_{q=0}^{\infty}\bigg(\Gamma_lR_l\xi_l^{q}\alpha(q)+\Gamma_lD_{l}\sqrt{m_l}n_l\sum_{r=0}^{q-2}\xi_l^{q-r-1}\alpha(r)\alpha(q)\bigg)\\ 
&\leq \frac{\Gamma_l}{1-\xi_l}R_l\alpha(0)+\Gamma_lD_{l}\sqrt{m_l}n_l\frac{\xi_l}{1-\xi_l}\sum_{q=0}^\infty \alpha^2(q)<\infty.
\end{aligned} \end{equation*}
Since $\sum_{q=0}^{\infty}\alpha(q)=\infty$ and $\sum_{q=0}^{\infty}\alpha^2(q)<\infty$, as $q$ tends to infinity, the sequence $\sigma_{l,i}(q\mathcal{R})$ converges to the average state $\bar{\sigma}_l(q\mathcal{R})$, while the surplus vector $s_{l,i}(q\mathcal{R})$ vanishes. 

We have analyzed the consensus for $t=q\mathcal{R}$, and in the following we consider the case for any $t\geq0$. For $t\in[q\mathcal{R},(q+1)\mathcal{R}-1)$ there exists a matrix $A\in\mathbb{R}^{n_l\times n_l}$ such that $\sigma_{l,i}(t)=\sum_{j=1}^{n_l} [A]_{ij}\sigma_{l,j}(q\mathcal{R}),$ where $\sum_{j=1}^{n_l}[A]_{ij}=1$ for any $i\in\{1,\dots,n_l\}$. Moreover, the average state $\bar{\sigma}_l(t)$ remains unchanged for $t\in[q\mathcal{R},(q+1)\mathcal{R}-1)$, $q\geq0$. Then 
\begin{equation*}
\lVert \bar{\sigma}_l(t)-\bar{\sigma}_l(t)\rVert \leq \sum_{j=1}^{n_l}[A]_{ij}\lVert \bar{\sigma}_l(q\mathcal{R})-\sigma_{l,j}(q\mathcal{R})\rVert\leq\max_{j}\lVert \bar{\sigma}_l(q\mathcal{R})-\sigma_{l,j}(q\mathcal{R})\rVert. 
\end{equation*}
Similarly, $\lVert s_{l,i}(t)\rVert\leq\max_j \lVert s_{l,i}(t)\rVert$ for $t\in[q\mathcal{R},(q+1)\mathcal{R}-1)$, $q\geq0$. Therefore, as $t$ tends to infinity, the sequence $\sigma_{l,i}(t)$ converges to the average state $\bar{\sigma}_l(t)$, and the surplus vectors $s_{l,i}(t)$ vanish.

Secondly, we show that the average state $\bar{\sigma}_{l,i}(t)$ converges to the DBNE in the zero-sum condition. So far, we have obtained an error bound between each agent's state $\sigma_{l,i}(t)$ and the average state $\bar{\sigma}_{l,i}(t)$. Since the average state $\bar{\sigma}_l(t)$ remains unchanged for $t\in[q\mathcal{R},(q+1)\mathcal{R}-1)$, $q\geq0$, we only need to show the convergence of $\bar{\sigma}_{l,i}(q\mathcal{R})$. According to the update rule, the update of the average state is 
\begin{equation*}
\bar{\sigma}_l((q+1)\mathcal{R})=\bar{\sigma}_l(q\mathcal{R})-\frac{\alpha(q)}{n_l}\sum_{i=1}^{n_l} g_{l,i}(q\mathcal{R}). 
\end{equation*}
Thus, 
\begin{equation*}\begin{aligned}
\lVert \bar{\sigma}_l((q+1)\mathcal{R})-\widetilde{\sigma}_l^*\rVert^2=&\lVert \frac{\alpha(q)}{n_l}\sum_{i=1}^{n_l} g_{l,i}(q\mathcal{R})\rVert^2+\lVert \bar{\sigma}_l(q\mathcal{R})-\widetilde{\sigma}_l^*\rVert^2\\&-\frac{2\alpha(q)}{n_l}\sum_{i=1}^{n_l}\left<\bar{\sigma}_l(q\mathcal{R})-\widetilde{\sigma}_l^*,g_{l,i}(q\mathcal{R})\right>.
\end{aligned} \end{equation*}
Consider the following Lyapunov candidate function 
\begin{equation*}
V(q)=\sum_{l=1}^2\lVert \bar{\sigma}_l(q\mathcal{R})-\widetilde{\sigma}_l^*\rVert^2. 
\end{equation*}
Then  
\begin{equation}\label{eq_lyp_update}
V(q+1)\leq V(q)+\left(\frac{D_1^2}{N_1}+\frac{D_2^2}{N_2}\right)\alpha^2(q)+\sum_{l=1}^2\frac{2\alpha(q)}{n_l}\sum_{i=1}^{n_l}\left<\widetilde{\sigma}_l^*-\bar{\sigma}_l(q\mathcal{R}),g_{l,i}(q\mathcal{R})\right>. 
\end{equation}
With Lemma \ref{lem_conv}, we only need to show that terms in \eqref{eq_lyp_update} satisfy the conditions in Lemma \ref{lem_conv}. It follows from $\sum_{q=0}^{\infty}\alpha^2(q)<\infty$ that the first term satisfies  
\begin{equation*}
\sum_{q=0}^{\infty} \left(\frac{D_1^2}{N_1}+\frac{D_2^2}{N_2}\right)\alpha^2(q)<\infty. 
\end{equation*}
Based on the property of subgradients, for $r\in\{1,\dots,N_l\}$,%
\begin{equation*}\begin{aligned}
&N_1\left<\widetilde{\sigma}_1^{*r}-\bar{\sigma}^r_1(q\mathcal{R}),g_{1,i}^r(q\mathcal{R})\right>\\
=&N_1\left(\left<\widetilde{\sigma}_1^{*r}-\sigma_{1,i}^r(q\mathcal{R}),g_{1,i}^r(q\mathcal{R})\right>+\left<\sigma_{1,i}^r(q\mathcal{R})-\bar{\sigma}_1^r(q\mathcal{R}),g_{1,i}^r(q\mathcal{R})\right>\right)\\
\leq& D_1\lVert \bar{\sigma}_1^r(q\mathcal{R})-\sigma_{1,i}^r(q\mathcal{R})\rVert +\widetilde{U}_{1,i}(\widetilde{\sigma}_1^*,\hat{\zeta}_{1,i}(q\mathcal{R}),\theta_1^r)-(\widetilde{U}_{1,i}(\sigma_{1,i}(q\mathcal{R}),\hat{\zeta}_{1,i}(q\mathcal{R}),\theta_1^r)+H_1(\sigma_{1,i}^r(q\mathcal{R})))\\
\leq& 2D_1\lVert \bar{\sigma}_1^r(q\mathcal{R})-\sigma_{1,i}^r(q\mathcal{R})\rVert +\widetilde{U}_{1,i}(\widetilde{\sigma}_1^*,\hat{\zeta}_{1,i}(q\mathcal{R}),\theta_1^r)-(\widetilde{U}_{1,i}(\bar{\sigma}_1(q\mathcal{R}),\hat{\zeta}_{1,i}(q\mathcal{R}),\theta_1^r)+H_1(\bar{\sigma}_1^r(q\mathcal{R}))).
\end{aligned}%
\end{equation*}
Moreover, by the Lipschitz continuity of $f_{1,i}(x_1,x_2,\theta_1,\theta_2)$ in $x_2\in\mathcal{X}_2$,%
\begin{equation*}
\begin{aligned}
&\widetilde{U}_{1,i}(\widetilde{\sigma}_1^*,\hat{\zeta}_{1,i}(q\mathcal{R}),\theta_1^r)-\widetilde{U}_{1,i}(\bar{\sigma}_l(q\mathcal{R}),\hat{\zeta}_{1,i}(q\mathcal{R}),\theta_1^r)\\
\leq&\widetilde{U}_{1,i}(\widetilde{\sigma}_1^*,\bar{\sigma}_2(q\mathcal{R}),\theta_1^r)-\widetilde{U}_{1,i}(\bar{\sigma}_1(q\mathcal{R}),\bar{\sigma}_2(q\mathcal{R}),\theta_1^r)+2L_{1,2}\lVert \hat{\zeta}_{1,i}(q\mathcal{R})-\bar{\sigma}_2(q\mathcal{R})\rVert.
\end{aligned}%
\end{equation*}
Since $\left<\widetilde{\sigma}_l^*-\bar{\sigma}_l(q\mathcal{R}),g_{l,i}(q\mathcal{R})\right>=\sum_{r=1}^{N_l}\left<\widetilde{\sigma}_l^{*r}-\bar{\sigma}^r_l(q\mathcal{R}),g_{1,i}^r(q\mathcal{R})\right>$, we can rewrite the last term of \eqref{eq_lyp_update} as%
\begin{align}\label{eq_lyp_lastterm}
&\sum_{l=1}^2 \frac{2\alpha(q)}{n_l}\sum_{i=1}^{n_l}\left<\widetilde{\sigma}_l^*-\bar{\sigma}_l(q\mathcal{R}),g_{l,i}(q\mathcal{R})\right>\\\notag\leq&
\sum_{l=1}^2\frac{2\alpha(q)}{N_ln_l}(2D_l\sqrt{N_l}\lVert \bar{\sigma}_l(q\mathcal{R})-\sigma_{l,i}(q\mathcal{R})\rVert+2N_lL_{l,3-l}\lVert \hat{\zeta}_{l,i}(q\mathcal{R})-\bar{\sigma}_{3-l}(q\mathcal{R})\rVert)\\\notag&+\frac{1}{N_1}\sum_{r=1}^{N_1}(\widetilde{U}_1(\widetilde{\sigma}_1^*,\bar{\sigma}_2(q\mathcal{R}),\theta_1^r)-\widetilde{U}_1(\bar{\sigma}_1(q\mathcal{R}),\bar{\sigma}_2(q\mathcal{R}),\theta_1^r)-H_1(\bar{\sigma}_1^r(q\mathcal{R})))\\\notag&+\frac{1}{N_2}\sum_{r=1}^{N_2}(\widetilde{U}_2(\bar{\sigma}_1(q\mathcal{R}),\sigma_2^*,\theta_2^r)-\widetilde{U}_2(\bar{\sigma}_1(q\mathcal{R}),\bar{\sigma}_2(q\mathcal{R}),\theta_2^r)-H_2(\bar{\sigma}_2^r(q\mathcal{R}))). 
\end{align}
According to Proposition \ref{prop_compress},%
\begin{equation*}
\lVert \hat{\zeta}_{3-l,i}(q\mathcal{R})-\bar{\sigma}_{l}(q\mathcal{R})\rVert\leq N_lm_l\lVert \bar{\sigma}_l(q\mathcal{R})-\sigma_{l,i}(q\mathcal{R})\rVert+\mathcal{S}D_l\sqrt{N_lm_l}\alpha(q).
\end{equation*}
Because $\sum_{q=0}^\infty \alpha(q)\lVert \bar{\sigma}_l(q\mathcal{R})-\sigma_{l,i}(q\mathcal{R})\rVert<\infty$ and $\sum_{q=0}^\infty \alpha^2(q)<\infty$, %
\begin{equation*}
\sum_{l=1}^2\frac{2\alpha(q)}{N_ln_l}(2D_l\sqrt{N_l}\lVert \bar{\sigma}_l(q\mathcal{R})-\sigma_{l,i}(q\mathcal{R})\rVert+2N_lL_{l,3-l}\lVert \hat{\zeta}_{l,i}(q\mathcal{R})-\bar{\sigma}_{3-l}(q\mathcal{R})\rVert)<\infty.%
\end{equation*}
Then we show that the remaining part of \eqref{eq_lyp_lastterm} is nonpositive.
Based on the choice of the discrete points \eqref{eq_dispoints}, the zero-sum condition in $\widetilde{G}$ is equivalent to %
\begin{equation}\label{eq_zerosum_dis}
\frac{1}{N_1}\sum_{i=1}^{N_1}\widetilde{U}_1(\widetilde{\sigma}_{1},\widetilde{\sigma}_{2},\theta_1^i)+\frac{1}{N_2}\sum_{j=1}^{N_2}\widetilde{U}_2(\widetilde{\sigma}_{1},\widetilde{\sigma}_{2},\theta_2^j)=0,\ \forall\,\widetilde{\sigma}_1\in\widetilde{\Sigma}_1,\,\forall\,\widetilde{\sigma}_2\in\widetilde{\Sigma}_2.%
\end{equation}
According to \eqref{eq_zerosum_dis}, denote the expectation of the cost function by  
\begin{equation*}
E\widetilde{U}(\sigma_1,\sigma_2)=\frac{1}{N_1}\sum_{r=1}^{N_1}\widetilde{U}_1(\sigma_1(\theta_1),\sigma_2(\theta_2),\theta_1^r)=-\frac{1}{N_2}\sum_{r=1}^{N_2}\widetilde{U}_2(\sigma_1(\theta_1),\sigma_2(\theta_2),\theta_1^r). 
\end{equation*}
Thus, the last two terms of \eqref{eq_lyp_lastterm} are expressed as 
\begin{align}\label{eq_lyp_zerosum}
&E\widetilde{U}(\widetilde{\sigma}_1^*,\bar{\sigma}_2(q\mathcal{R}))-E\widetilde{U}(\bar{\sigma}_1(q\mathcal{R}),\bar{\sigma}_2(q\mathcal{R}))-\frac{1}{N_1}\sum_{r=1}^{N_1}H_1(\bar{\sigma}^r_1(q\mathcal{R}))\\\notag&-(E\widetilde{U}(\bar{\sigma}_1(q\mathcal{R}),\widetilde{\sigma}_2^*)-E\widetilde{U}(\bar{\sigma}_1(q\mathcal{R}),\bar{\sigma}_2(q\mathcal{R}))-\frac{1}{N_2}\sum_{r=1}^{N_2}H_2(\bar{\sigma}^r_2(q\mathcal{R}))\\
=\notag&E\widetilde{U}(\widetilde{\sigma}_1^*,\widetilde{\sigma}_2^*) - E\widetilde{U}(\bar{\sigma}_1(q\mathcal{R}),\widetilde{\sigma}_2^*)-\frac{1}{N_1}\sum_{r=1}^{N_1}H_1(\bar{\sigma}^r_1(q\mathcal{R}))\\
\notag& -(E\widetilde{U}(\widetilde{\sigma}_1^*,\widetilde{\sigma}_2^*)-E\widetilde{U}(\widetilde{\sigma}_1^*,\bar{\sigma}_2(q\mathcal{R})))-\frac{1}{N_2}\sum_{r=1}^{N_2}H_2(\bar{\sigma}^r_2(q\mathcal{R})). 
\end{align}
According to the definition of the penalty function, $H_l(\sigma_l^*)=0$. From Proposition \ref{prop_penalty}(c), a minimizer of $\widetilde{U}_1$ in $\mathcal{X}_1$ is also the minimizer of $\widetilde{U}_1+H_1$ in $\mathbb{R}^{m_1}$, \textit{i.e.}, for all $\bar{\sigma}_1(q\mathcal{R})\in\mathbb{R}^{N_1m_1}$ and $r=1,\dots,N_1$,  
\begin{equation*}
\begin{aligned}
&\widetilde{U}_1(\widetilde{\sigma}_1^*,\widetilde{\sigma}_2^*,\theta_1^r)-\widetilde{U}_1(\bar{\sigma}_1(q\mathcal{R}),\widetilde{\sigma}_2^*,\theta_1^r)-H_1(\bar{\sigma}^r_1(q\mathcal{R})\\
=&\widetilde{U}_1(\widetilde{\sigma}_1^*,\widetilde{\sigma}_2^*,\theta_1^r)+H_1(\sigma_1^{*r})-(\widetilde{U}_1(\bar{\sigma}_1(q\mathcal{R}),\widetilde{\sigma}_2^*,\theta_1^r)+H_1(\bar{\sigma}^r_1(q\mathcal{R}))\leq0.
\end{aligned} \end{equation*}
Up to now, we have proved that \eqref{eq_lyp_zerosum} is negative. Therefore, with Lemma \ref{lem_conv}, the Lyapunov function $V(q)$ converges to a finite number. Furthermore, 
\begin{equation*}\begin{aligned}
0\leq \sum_{q=0}^{\infty}\alpha(q)\bigg(&\frac{1}{N_1}\sum_{r=1}^{N_1}(\widetilde{U}_1(\bar{\sigma}_1(q\mathcal{R}),\widetilde{\sigma}_2^*,\theta_1^r)+H_1(\bar{\sigma}^r_1(q\mathcal{R}))-\widetilde{U}_1(\widetilde{\sigma}_1^*,\widetilde{\sigma}_2^*,\theta_1^r))\\
+&\frac{1}{N_2}\sum_{r=1}^{N_2}(\widetilde{U}_2(\widetilde{\sigma}_1^*,\bar{\sigma}_2(q\mathcal{R}),\theta_2^r)+H_2(\bar{\sigma}_2^r(q\mathcal{R}))-\widetilde{U}_2(\widetilde{\sigma}_1^*,\widetilde{\sigma}_2^*,\theta_2^r))\bigg)<\infty.\end{aligned} 
\end{equation*}

Then we prove the convergence to the DBNE by the strict convexity of $f_l$. Because $\sum_{t=0}^{\infty}\alpha(t)=\infty$, there exists a subsequence $\{q_s\}$ such that 
\begin{equation*}
\lim_{s\to\infty}\sum_{r=1}^{N_1}(\widetilde{U}_1(\bar{\sigma}_1(q_s\mathcal{R}),\widetilde{\sigma}_1^*,\theta_1^r)+H_1(\bar{\sigma}_1^r(q_s\mathcal{R})))=\sum_{r=1}^{N_1}\widetilde{U}_1(\widetilde{\sigma}_1^*,\widetilde{\sigma}_1^*,\theta_1^r). 
\end{equation*}
Denote the limit point by $\bar{\bar{\sigma}}_1=\lim_{s\to\infty}\bar{\sigma}_1(q_s\mathcal{R})$. As a result, 
\begin{equation*}
\sum_{r=1}^{N_1}(\widetilde{U}_1(\bar{\bar{\sigma}}_1,\widetilde{\sigma}_1^*,\theta_1^r)+H_1(\bar{\bar{\sigma}}_1^r))=\sum_{r=1}^{N_1}\widetilde{U}_1(\widetilde{\sigma}_1^*,\widetilde{\sigma}_1^*,\theta_1^r), 
\end{equation*}
By the strict convexity of $\widetilde{U}_1$ and Proposition \ref{prop_penalty}(c), $
\lim_{s\to\infty}\bar{\sigma}_1(q_s\mathcal{R})=\widetilde{\sigma}_1^*.$ Similarly, $
\lim_{s\to\infty}\bar{\sigma}_2(q_s\mathcal{R})=\widetilde{\sigma}_2^*.$ With $\lim_{t\to\infty}\lVert \sigma_{l,i}(t)-\bar{\sigma}_l(t)\rVert=0$, 
\begin{equation*}
\lim_{s\to\infty}\sigma_{l,i}(q_s\mathcal{R})=\lim_{s\to\infty}\bar{\sigma}_l(q_s\mathcal{R})=\widetilde{\sigma}_l^*,\ l\in\{1,2\}, 
\end{equation*} 
which implies that there exists a subsequence that converges to the DBNE. According to Lemma \ref{thm_bne}, the error between the DBNE of $\widetilde{G}$ and the BNE of $G$ is bounded by $\epsilon=O(\max_{l\in\{1,2\},i\in\{1,\dots,N_l\}}\{\theta_l^i-\theta_1^{i-1}\})$. Thus, we complete the proof of Theorem \ref{thm_conv}.
\end{proof}

With the guarantee of the convergence, we next prove that Algorithm \ref{alg_main} attains an $O(\ln T/\sqrt{T})$ convergence rate, namely Theorem \ref{thm_conv_rate}.
\begin{proof}[Proof of Theorem \ref{thm_conv_rate}]
Take $T_0=\lfloor T/\mathcal{R}\rfloor=O(T)$. Since the column of $\bar{M}_{l,k}(t)$ sums up to 1 for any $t\geq0$, $\bar{\sigma}_l(q\mathcal{R}+t')=\bar{\sigma}_l(q\mathcal{R})$ for all $t'=0,\dots,\mathcal{R}-1$. Define %
\begin{equation*}
E\widetilde{U}_{\max}(\sigma_1^*)=\max_{1\leq q\leq T_0}E\widetilde{U}(\sigma_1^*,\bar{\sigma}_2(q\mathcal{R})),\ E\widetilde{U}_{\min}(\sigma_2^*)=\min_{1\leq q\leq T_0}E\widetilde{U}(\bar{\sigma}_1(q\mathcal{R}),\sigma_2^*).%
\end{equation*} Thus, according to \eqref{eq_con_rate}, %
\begin{align}\label{eq_converrate}
\notag(E\widetilde{U}_{\min}(\sigma_2^*)-E\widetilde{U}(\sigma_1^*,\sigma_2^*))\sum_{q=0}^{T_0}\alpha(q)&\leq\sum_{q=0}^{T_0}(E\widetilde{U}(\bar{\sigma}_1(q\mathcal{R}),\sigma_2^*)-E\widetilde{U}(\sigma_1^*,\sigma_2^*))\alpha(q)\\
&\leq L_{1,1}\sum_{q=0}^{T_0}\alpha(q)Y_1^1(q)\leq C_2+C_3\sum_{q=0}^{T_0}\alpha(q).
\end{align}
where $C_2=\frac{L_{1,1}\Gamma_1R_1\alpha(0)}{1-\xi_1}$ and $C_3=L_{1,1}D_{1}\sqrt{N_1m_1}\left( \frac{\Gamma_1\xi_1}{1-\xi_1}+2\right)$. Then \eqref{eq_converrate} can be equivalently expressed as 
\begin{equation}\label{eq_converrate1}
E\widetilde{U}_{\min}(\sigma_2^*)-E\widetilde{U}(\sigma_1^*,\sigma_2^*)\leq \frac{C_2}{\sum_{q=0}^{T_0}\alpha(q)}+\frac{C_3\sum_{q=0}^{T_0}\alpha(q)^2}{\sum_{q=0}^{T_0}\alpha(q)}. 
\end{equation}
With $\alpha(q)=1/\sqrt{q}$, the two terms on the right hand of \eqref{eq_converrate1} satisfy 
\begin{equation*}
\frac{C_2}{\sum_{q=0}^{T_0}\alpha(q)}=\frac{C_2/2}{\sqrt{T_0}-1}=O\left(\frac{1}{\sqrt{T}}\right),\ \frac{C_3\sum_{q=0}^{T_0}\alpha(q)^2}{\sum_{q=0}^{T_0}\alpha(q)}=\frac{C_3\ln T_0}{2(\sqrt{T_0}-1)}=O\left(\frac{\ln T}{\sqrt{T}}\right). 
\end{equation*}
Similarly, $E\widetilde{U}(\sigma_1^*,\sigma_2^*)-E\widetilde{U}_{\max}(\sigma_1^*)=O(\ln T/\sqrt{T})$. Thus, we finish the proof of Theorem \ref{thm_conv_rate}.
\end{proof}

\section{Numerical simulations}\label{sec_simu}
In this section, we provide numerical simulations to illustrate the effectiveness of Algorithm \ref{alg_main} on subnetwork zero-sum Bayesian games.

Consider a symmetric rent-seeking game with two subnetworks who aim to choose a level of costly effort in order to obtain a share of a prize \cite{fey2008,guow2021}. Each subnetwork $\Xi_l$ consists of three agents. The feasible action sets satisfy $\mathcal{X}_1=\mathcal{X}_2=[0.1,1]$ and $\Theta_1=\Theta_2=[0.01,1.01]$. Also, $\theta_1$ and $\theta_2$ are independent and uniformly distributed over $\Theta_1$ and $\Theta_2$, respectively. For $l\in\{1,2\}$, the cost functions of agents are as follows.%
\begin{equation*}
\begin{aligned}
f_{l,1}(x_1,x_2,\theta_1,\theta_2)&=\frac{(x_l-x_{3-l})(\theta_1+\theta_2)}{2}-\frac{x_l}{6(x_l+x_{3-l})},\\
f_{l,2}(x_1,x_2,\theta_1,\theta_2)&=\frac{(x_l-x_{3-l})(\theta_1+\theta_2)}{2}-\frac{x_l}{2(x_l+x_{3-l})},\\
f_{l,3}(x_1,x_2,\theta_1,\theta_2)&=-\frac{x_l}{3(x_l+x_{3-l})}.\end{aligned}%
\end{equation*}
The communication graph switches periodically over the two graphs $\mathcal{G}^e,\mathcal{G}^o$ given in Fig. \ref{graphs}, where $\mathcal{G}(2k)=\mathcal{G}^e$ and $\mathcal{G}(2k+1)=\mathcal{G}^o$, $k\geq 0$. 

\begin{figure}[h]
\centerline{\includegraphics[width=0.5\textwidth]{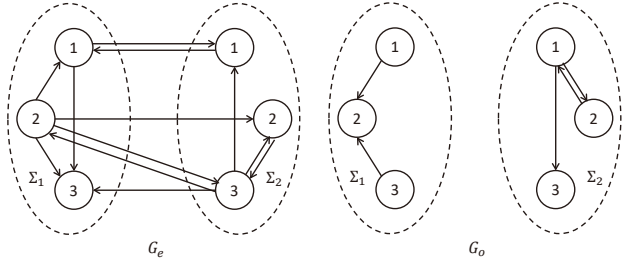}}
\caption{Two possible communication graphs}\label{graphs}
\end{figure}

We take the number of discrete points $N_1=N_2=N$ and the compression ratio $\rho_1=\rho_2=\rho$ in Algorithm \ref{alg_main}. Firstly, we illustrate the convergence of Algorithm \ref{alg_main}. We present the trajectories of strategies in $\Xi_1$ for types $\theta_1=0.1$ and $0.8$ under $N=1000$ and $\rho=0.5$ in Fig. \ref{fig_converge_agents}. We can see that the agents' strategies reach consensus and converge. Fig. \ref{fig_converge_N} shows strategy trajectories of agent 1 in $\Xi_l$ for types $\theta_1=0.1$ and $0.8$ under different $N$. We find that the limit point of $\sigma_{1,1}(t)$ converges as $N$ tends to infinity, which is consistent with Theorem \ref{thm_conv}.
\begin{figure}[!h]
\centering
\subfloat[$\theta_1=0.1$]{\includegraphics[width=0.43\textwidth]{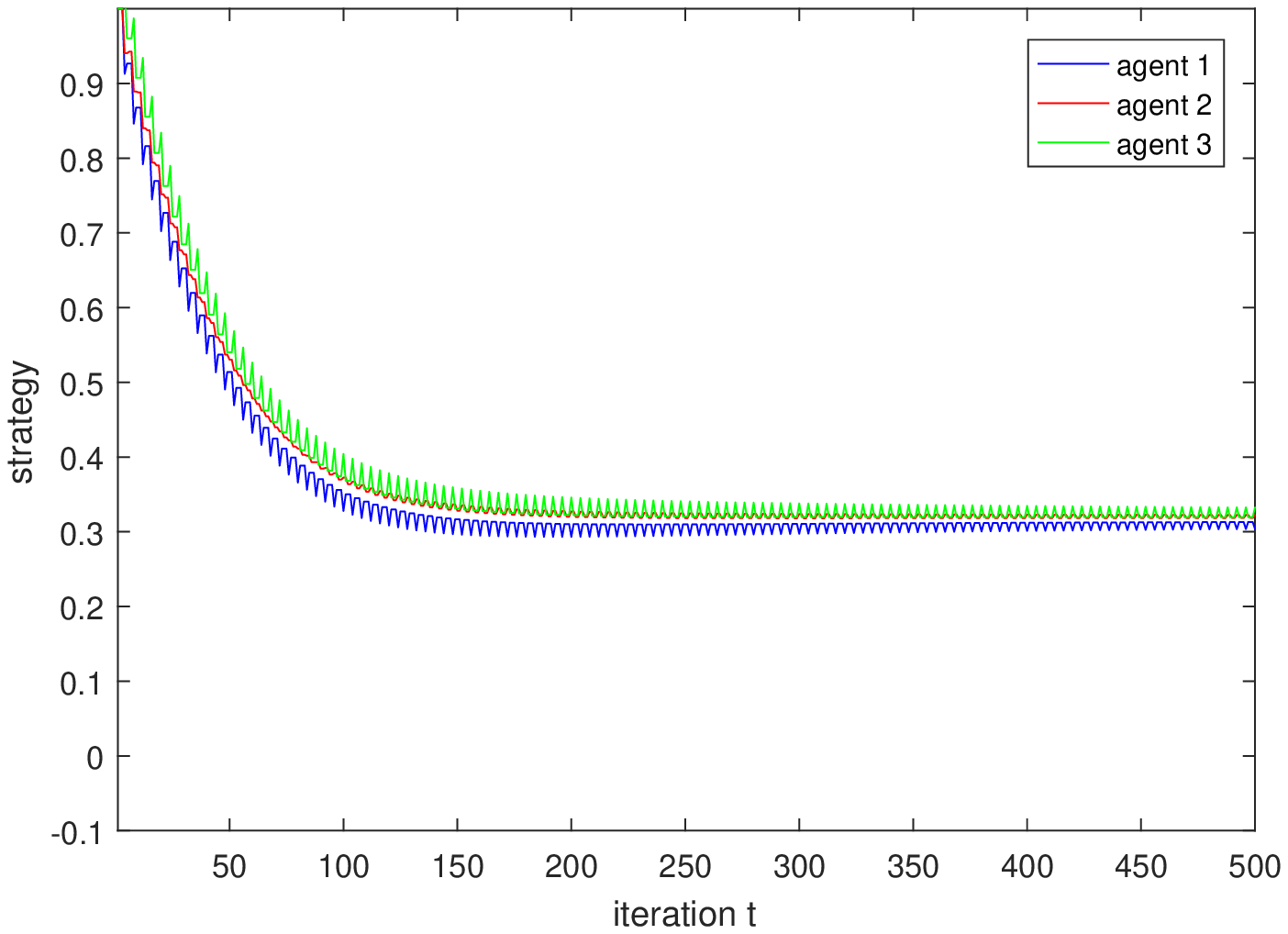}}
\hfill
\subfloat[$\theta_1=0.8$]{\includegraphics[width=0.43\textwidth]{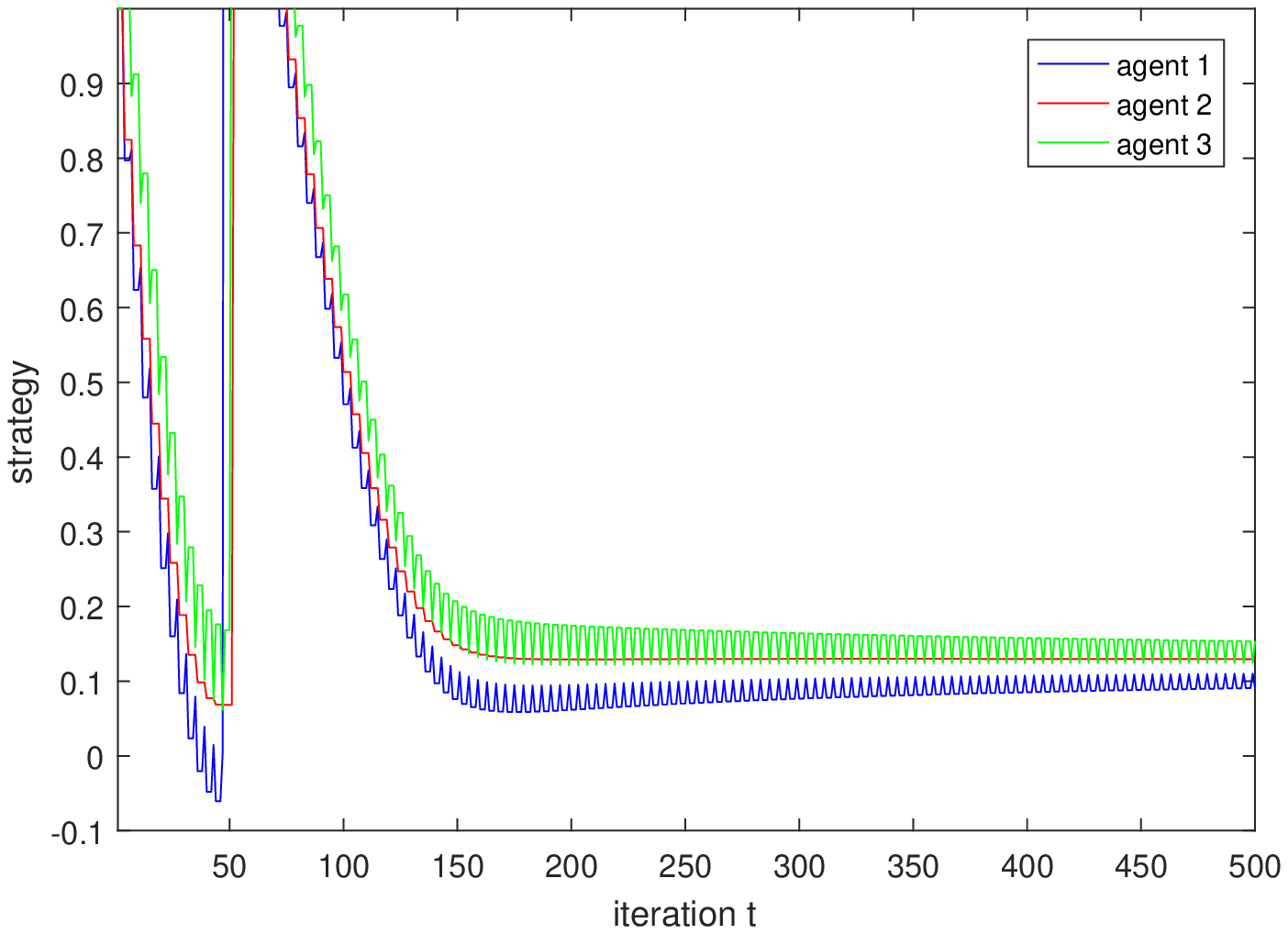}}
\caption{Strategies of agents in $\Xi_1$ at $\theta_1=0.1$ and $0.8$ under $N=1000$ and $\rho=0.5$.}\label{fig_converge_agents}
\end{figure}
\begin{figure}[!h]
\centering
\subfloat[$\theta_1=0.1$]{\includegraphics[width=0.43\textwidth]{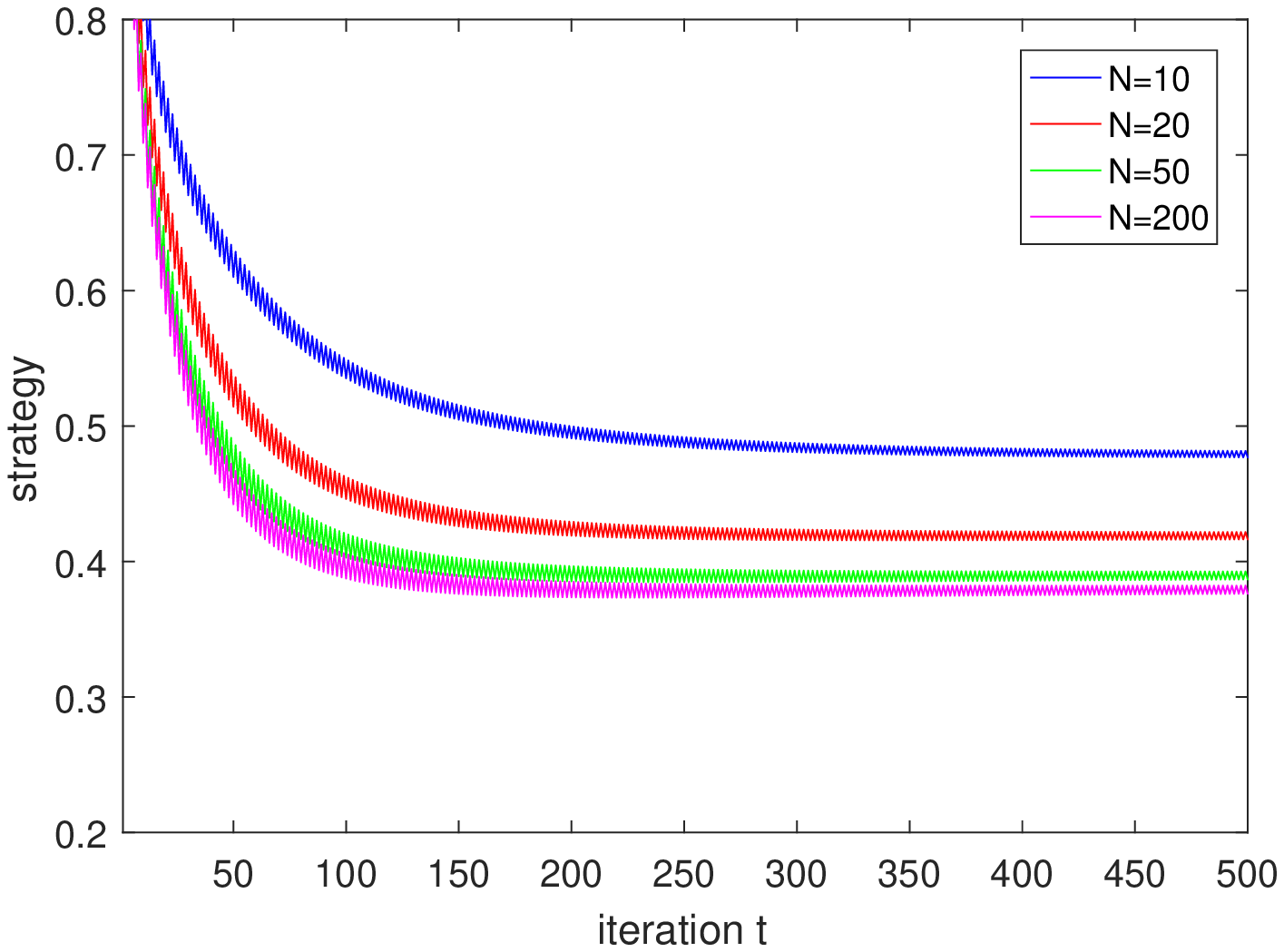}}
\hfill
\subfloat[$\theta_1=0.8$]{\includegraphics[width=0.43\textwidth]{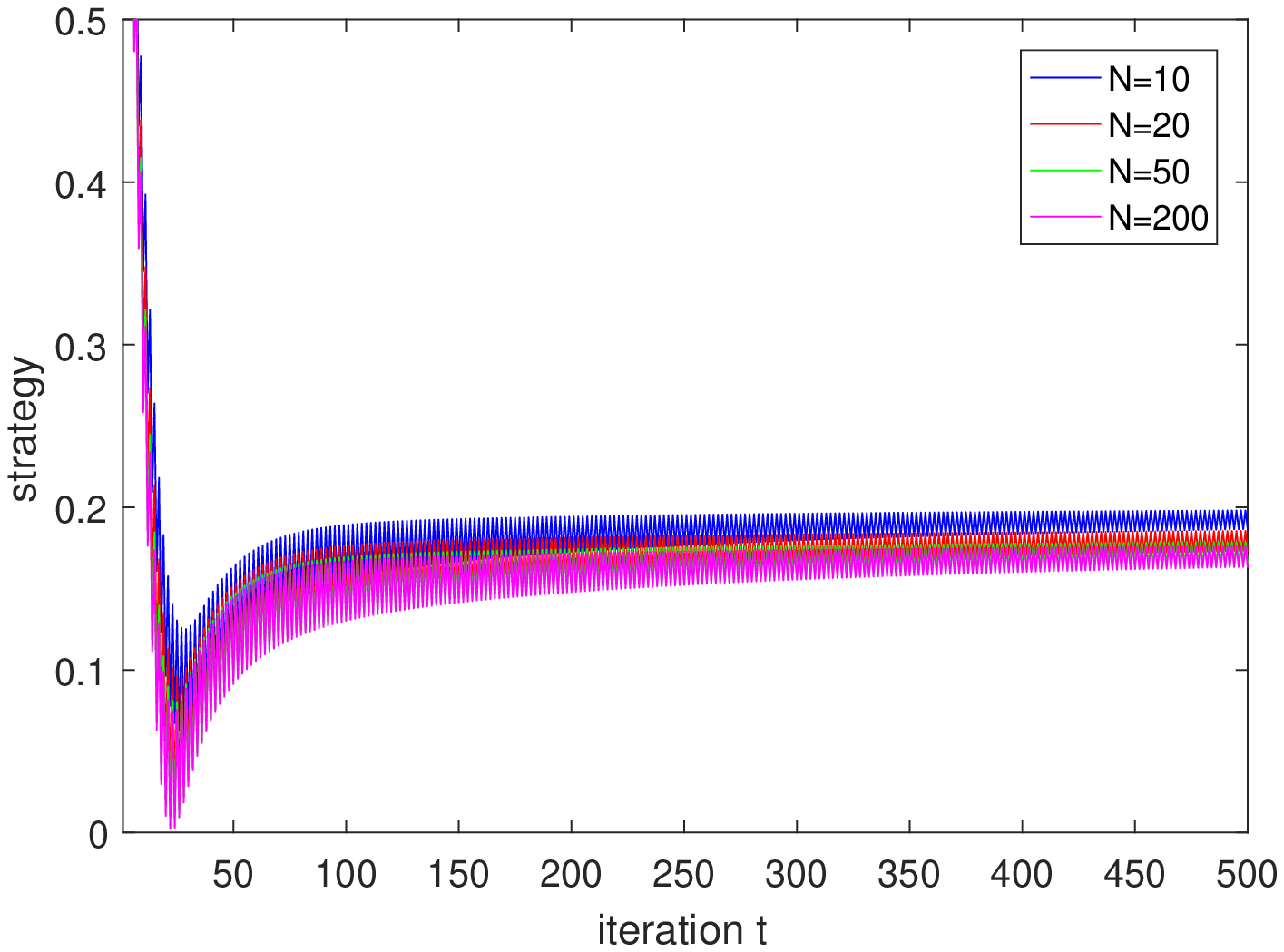}}
\caption{Strategies of agent $1$ in $\Xi_1$ at $\theta_1=0.1$ and $0.8$ with different $N=1000$ under $\rho=1$.}\label{fig_converge_N}
\end{figure}

Then we demonstrate the effectiveness of compression. In the communication, each agent sends $d$-dimensional compressed strategies (8 bytes per element) and the index of those strategies (4 bytes per element) to their current neighbors. Thus, the communication data size of an iteration can be calculated by  
\begin{equation}\label{eq_datasize}
\mathrm{Data\ size}=12\times d\times \mathrm{communication\ times}. 
\end{equation}
\eqref{eq_datasize} matches the result of our simulations in Table \ref{tab_datasize}. \begin{table}[t]
{\footnotesize
\caption{Average data size in network communication with different $N$ and $\rho$.}\label{tab_datasize}
\begin{center}
\begin{tabular}{|c|c|c|c|c|} \hline
Data (kb) & $\rho=1$ & $\rho=0.5$ & $ \rho=0.2$ & $\rho=0.1$\\ \hline
$N=200$ & 19.92 &9.96 &3.98 & 1.99\\
$N=1000$ & 99.61 & 49.80 & 19.92 & 9.96\\ \hline
\end{tabular}
\end{center}
}
\end{table}
We can see that our compression considerably reduce the communication loads. Fig. \ref{fig_converge_compress} demonstrates the convergence of strategies in $\Xi_1$ for specific types $\theta_1=0.1$ and $0.8$ under different $\rho$ with $N=1000$. As Fig. \ref{fig_converge_compress}, the compression does not affect the convergence, which means that whatever $\rho$ we choose, Algorithm \ref{alg_main} remains convergent.

\begin{figure}[h]
\centering
\subfloat[$\theta_1=0.1$]{\includegraphics[width=0.43\textwidth]{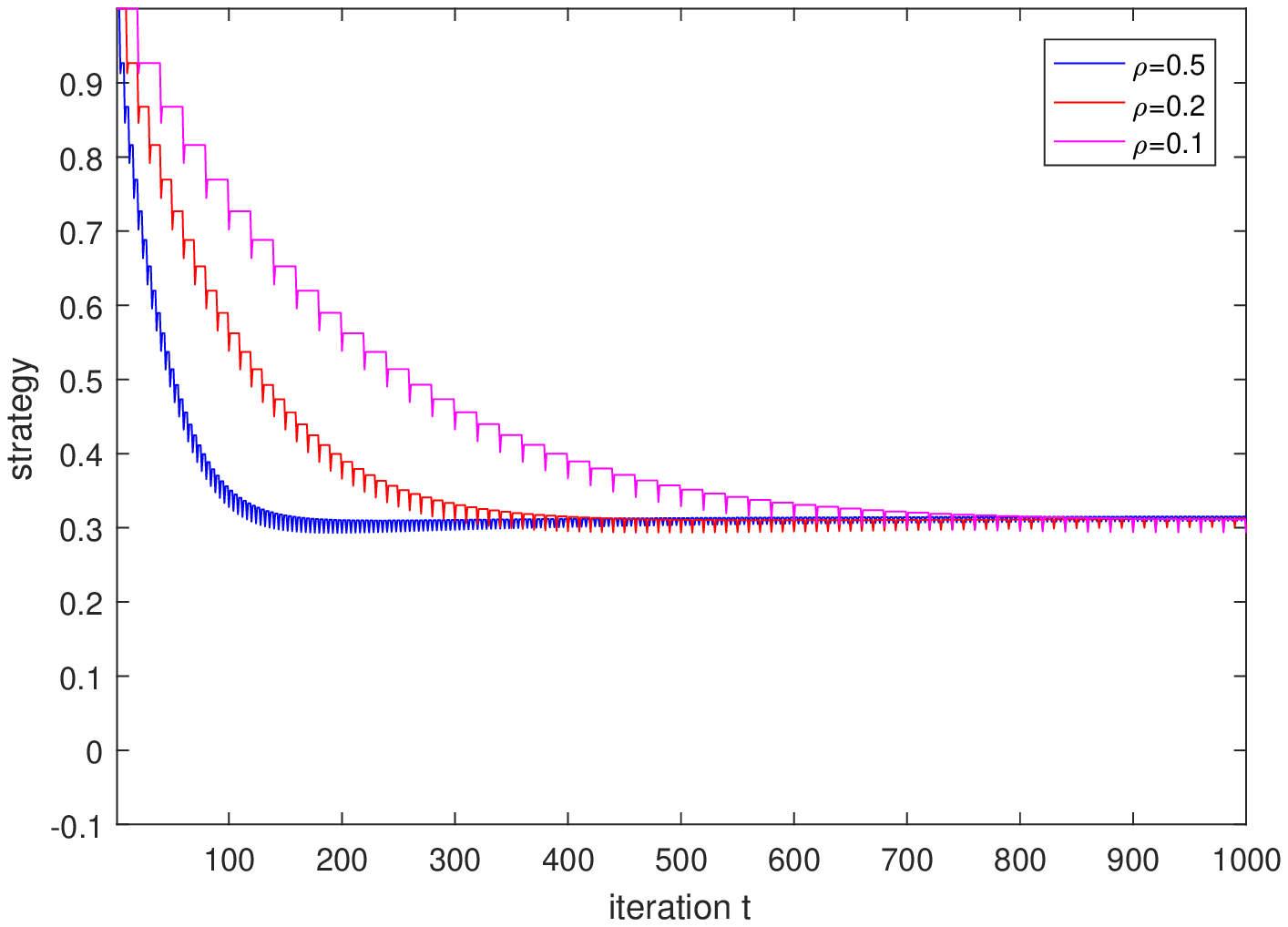}}
\hfill
\subfloat[$\theta_1=0.8$]{\includegraphics[width=0.43\textwidth]{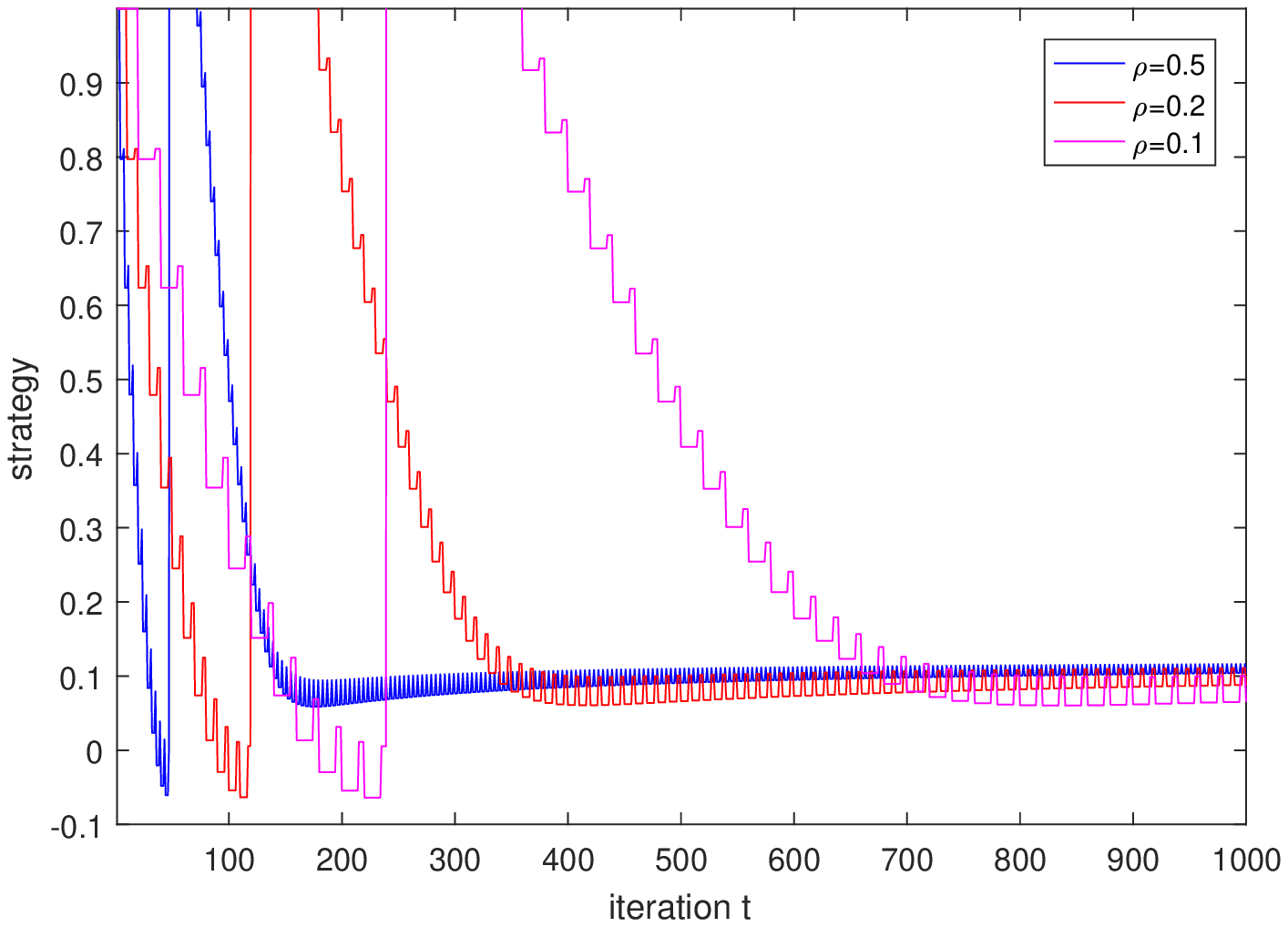}}
\caption{Strategies of agent $1$ in $\Xi_1$ at $\theta_1=0.1$ and $0.8$ with different $\rho$ under $N=1000$.}\label{fig_converge_compress}
\end{figure}

\begin{figure}[!h]
\centering
\subfloat{\label{fig:bb}\includegraphics[width=0.43\textwidth]{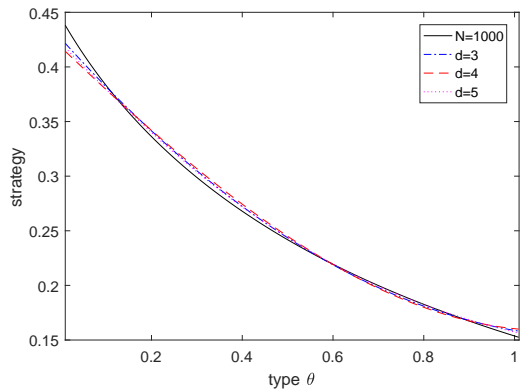}}
\caption{The approximate equilibrium generated by Algorithm \ref{alg_main} ($N=1000$) and by polynomial approximations in {\rm\cite{guos2021}} ($d=3,4,5$). }\label{fig_appro}
\end{figure}

Next, we verify the approximation of the BNE generated by Algorithm \ref{alg_main}. Here we compare the equilibrium obtained from Algorithm \ref{alg_main} and the equilibria obtained by polynomial approximation in \cite{guos2021} in Fig. \ref{fig_appro}. Compared with results by polynomial approximations, we believe that the DBNE is close to the true equilibrium.

\section{Conclusion}\label{sec_con}
In this paper, we proposed a distributed algorithm for seeking a continuous-type BNE in subnetwork zero-sum games. We showed that the algorithm could obtain an approximate BNE with an explicit error bound via communication-efficient computation. Our algorithm involved two main steps. In the discretization step, we established a discretized model and provided the relation between the DBNE of the discretized model and the BNE of the continuous model with the explicit error bound. Then in the communication step, we provided a novel communication scheme with a designed sparsification rule, which could effectively reduce the amount of communication and adapt well to time-varying networks, and we proved that agents could obtain unbiased estimations through such communication. Finally, we provided convergence analysis of the algorithm and its convergence rate in the considered settings.

\appendix
\renewcommand\thesection{Appendix~\Alph{section}}

\section{Proof of Lemma \ref{prop_br}}\label{pf_br}
For any $i\in\{1,\dots,N_l\}$ and $\theta_{3-l}\in\Theta_{3-l}$, according to L'Hospital's rule, 
\begin{equation*}
\lim_{N_l\to\infty}\frac{\int_{\theta_l^{i-1}}^{\theta_l^i}p(\theta_1,\theta_2)d\theta_l}{\int_{\theta_l^{i-1}}^{\theta_l^i}p_l(\theta_l)d\theta_l}=\frac{p(\theta_1,\theta_2)|_{\theta_l=\theta_l^i}}{p_l(\theta_l^i)}=p_l(\theta_{3-l}|\theta_l^i).
\end{equation*}
Thus, by the best responses in $G$ and $\widetilde{G}$, for any $\widetilde{\sigma}_{l*}\in BR_l^{N_l}(\sigma_{3-l})$ and $\theta_l\in\widetilde{\Theta}_l$, there exists a strategy $\sigma_{l*}\in BR_l(\sigma_{3-l})$ such that, as $N_l\to\infty$, $\widetilde{\sigma}^{N_l}_{l*}(\theta_l)=\sigma_{l*}(\theta_l).$ Since $\sigma_{l*}$ is piecewise continuous, for any $\varepsilon>0$, there exists a $\delta>0$ such that, for any $\theta_l\in\Theta_l$ except for finite points and $\theta_l'\in B(\theta_l,\delta)\cap\Theta_l$, $|\sigma_{l*}(\theta_l)-\sigma_{l*}(\theta_l')|<\varepsilon$. As $N_l\to\infty$, $\max_{i\in\{1,\dots,N_l\}} (\theta_l^i-\theta_l^{i-1})<\delta$, 
\begin{equation}\label{eq_br}
\int_{\Theta_l} \lVert\widetilde{\sigma}^{N_l}_{l*}(\theta_l)-\sigma_{l*}(\theta_l)\rVert^2d\theta_l\leq\sum_{i=1}^{N_l}\varepsilon^2(\overline{\theta}_l^i-\underline{\theta}_l^{i-1})=\varepsilon^2(\overline{\theta}_l-\underline{\theta}_l). 
\end{equation}
As $\varepsilon\to0$, \eqref{eq_br} implies $\widetilde{\sigma}^{N_l}_{l*}(\theta_l)-\sigma_{l*}(\theta_l)\to0$ for almost every $\theta_l\in\widetilde{\Theta}_l$. 

\section{Proof of Lemma \ref{thm_bne}}\label{pf_bne}
To estimate the BNE of $G$ with strategies in $\widetilde{G}$, we first convert the continuous model $G$ to the discretized model $\widetilde{G}$, and then estimate the error between $\int_{\theta_l}U_lp_l(\theta_l)$ and $\sum_{\theta_l}\widetilde{U}_l\widetilde{P}_l(\theta_l)$ for strategies in $\widetilde{\Sigma}_l$.

From Assumption \ref{ass_game}(iv), for each $l\in\{1,2\}$ and $k\in\{1,2\}$, $f_l$ is $L_{\theta}$ continuous in $\theta_k\in\Theta_k$ for any $x_1\in\mathcal{X}_1$, $x_2\in\mathcal{X}_2$ and $\theta_{3-k}\in\Theta_{3-k}$. For any $\theta_1\in(\theta_1^{i-1},\theta_1^i]$ and $\theta_2\in(\theta_2^{j-1},\theta_2^j]$, $$|f_l(x_1,x_2,\theta_1,\theta_2)-f_l(x_1,x_2,\theta_1^i,\theta_2^j)|\leq L_{\theta}|\theta_1^i-\theta_1|+L_{\theta}|\theta_2^j-\theta_2|\leq 2L_{\theta}\epsilon_0,$$ where $\epsilon_0=\max\{\theta_1^{i}-\theta_1^{i-1},\theta_2^{j}-\theta_2^{j-1}\}$. Note that a strategy in $\widetilde{\Sigma}_l$ is a constant for the types lying in $(\theta_l^{i-1},\theta_l^i]$. For any $\widetilde{\sigma}_1\in\widetilde{\Sigma}_1$ and $\widetilde{\sigma}_2\in\widetilde{\Sigma}_2$,
\begin{align}\label{eq_transinprf}
&\int_{\underline{\theta}_1}^{\overline{\theta}_1}\int_{\underline{\theta}_2}^{\overline{\theta}_2} f_l(\widetilde{\sigma}_1(\theta_1),\widetilde{\sigma}_2(\theta_2),\theta_1,\theta_2)p(\theta_1,\theta_2)d\theta_1d\theta_2\\
\notag= &\sum_{i=1}^{N_1}\sum_{j=1}^{N_2}\int_{\theta_1^{i-1}}^{\theta_1^i}\int_{\theta_2^{j-1}}^{\theta_2^j} \big(f_l(\widetilde{\sigma}_1(\theta_1),\widetilde{\sigma}_2(\theta_2),\theta_1^i,\theta_2^j)+f_l(\widetilde{\sigma}_1(\theta_1),\widetilde{\sigma}_2(\theta_2),\theta_1,\theta_2)\\\notag&-f_l(\widetilde{\sigma}_1(\theta_1),\widetilde{\sigma}_2(\theta_2),\theta_1^i,\theta_2^j)\big)p(\theta_1,\theta_2)d\theta_1d\theta_2\\
\notag\geq &\sum_{i=1}^{N_1}\sum_{j=1}^{N_2}f_l(\widetilde{\sigma}_1(\theta_1),\widetilde{\sigma}_2(\theta_2),\theta_1^i,\theta_2^j)\widetilde{P}(\theta_1^i,\theta_2^j)+2L_\theta\epsilon_0.
\end{align}
Then, according to the definition of the DBNE in \eqref{eq_bne_dis}, for any $i\in\{1,\dots,N_1\}$, 
\begin{equation}\label{eq_disbne_inq}
\sum_{j=1}^{N_2}f_1(x_1,\widetilde{\sigma}_2^*(\theta_2),\theta_1^i,\theta_2^j)\widetilde{P}_1(\theta_2^j|\theta_1^i)\geq\sum_{j=1}^{N_2}f_1(\widetilde{\sigma}_1^*(\theta_1^i),\widetilde{\sigma}_2^*(\theta_2^j),\theta_1^i,\theta_2^j)\widetilde{P}_1(\theta_2^j|\theta_1^i). 
\end{equation}

Our purpose is to convert the term on the left hand of \eqref{eq_disbne_inq} from the discretized form to the continuous form. To get the relation between the discretized form and the continuous form, we adopt the distribution used in Definition \ref{def_disbr}. For any $\sigma_1\in\Sigma_1$, 
\begin{align}\label{eq_contintodis}
&\sum_{j=1}^{N_2}\int_{\theta_1^{i-1}}^{\theta_1^i} f_1(\sigma_1(\theta_1),\widetilde{\sigma}_2^*(\theta_2^j),\theta_1^i,\theta_2^j)\int_{\theta_2^{j-1}}^{\theta_2^j}p(\theta_1,\theta_2)d\theta_1d\theta_2\\
=&\notag\sum_{j=1}^{N_2}\int_{\theta_1^{i-1}}^{\theta_1^i}f_1(\sigma_1(\theta_1),\widetilde{\sigma}_2^*(\theta_2^j),\theta_1^i,\theta_2^j)\int_{\theta_2^{j-1}}^{\theta_2^j}\bigg(p(\theta_1,\theta_2)-\frac{\int_{\theta_1^{i-1}}^{\theta_1^i}p(\theta_1',\theta_1)d\theta_1'}{\theta_1^i-\theta_1^{i-1}}\\&\notag\qquad\qquad\qquad\qquad\qquad\qquad\qquad\qquad\quad\ \ +\frac{\int_{\theta_1^{i-1}}^{\theta_1^i}p(\theta_1',\theta_1)d\theta_1'}{\theta_1^i-\theta_1^{i-1}}\bigg)d\theta_1d\theta_2.
\end{align}
Due to the Lipschitz continuity of $p(\theta_1,\theta_2)$,%
\begin{align}\label{eq_probcontintodis}
\left|p(\theta_1,\theta_2)-\frac{\int_{\theta_1^{i-1}}^{\theta_1^i}p(\theta_1',\theta_1)d\theta_1'}{\theta_1^i-\theta_1^{i-1}}\right|&\leq\frac{1}{\theta_1^i-\theta_1^{i-1}}\int_{\theta_1^{i-1}}^{\theta_1^i}|p(\theta_1,\theta_2)-p(\theta_1',\theta_1)|d\theta_1'\\
\notag&\leq \frac{1}{\theta_1^i-\theta_1^{i-1}}\int_{\theta_1^{i-1}}^{\theta_1^i}L_p|\theta_1-\theta_1'|d\theta_1'
\leq L_p\epsilon_0.
\end{align}
According to the Lipschitz continuity of $f_l$ and the compactness of $\mathcal{X}_1$, $\mathcal{X}_2$, $\Theta_1$, and $\Theta_2$, there exists a constant $M$ such that $|f_l|\leq M$. Applying \eqref{eq_probcontintodis} to \eqref{eq_contintodis},  
\begin{equation*}
\begin{aligned}
&\sum_{j=1}^{N_2}\int_{\theta_1^{i-1}}^{\theta_1^i} f_1(\sigma_1(\theta_1),\widetilde{\sigma}_2^*(\theta_2^j),\theta_1^i,\theta_2^j)\int_{\theta_2^{j-1}}^{\theta_2^j}p(\theta_1,\theta_2)d\theta_1d\theta_2\\
\geq&\sum_{j=1}^{N_2}\int_{\theta_1^{i-1}}^{\theta_1^i} f_1(\sigma_1(\theta_1),\widetilde{\sigma}_2^*(\theta_2^j),\theta_1^i,\theta_2^j)\frac{d\theta_1}{\theta_1^i-\theta_1^{i-1}}\widetilde{P}(\theta_1^i,\theta_2^j)-ML_p(\theta_1^i-\theta_1^{i-1})(\overline{\theta}_2-\underline{\theta}_2)\epsilon_0\\
\geq&\sum_{j=1}^{N_2}f_1(\widetilde{\sigma}_1^*(\theta_1^i),\widetilde{\sigma}_2^*(\theta_2^j),\theta_1^i,\theta_2^j)\widetilde{P}(\theta_1^i,\theta_2^j)-ML_p(\theta_1^i-\theta_1^{i-1})(\overline{\theta}_2-\underline{\theta}_2)\epsilon_0.
\end{aligned}%
\end{equation*}
Thus, for any $\sigma_1\in\Sigma_1$, \eqref{eq_disbne_inq} can be rewritten as%
\begin{align}\label{eq_pfap}
&\sum_{i=1}^{N_1}\sum_{j=1}^{N_2}f_1(\widetilde{\sigma}^*_1(\theta_1^i),\widetilde{\sigma}^*_2(\theta_2^j),\theta_1^i,\theta_2^j)\widetilde{P}(\theta_1^i,\theta_2^j)\\
\notag\leq&\sum_{i=1}^{N_1}\sum_{j=1}^{N_2}\int_{\theta_1^{i-1}}^{\theta_1^i}\int_{\theta_2^{j-1}}^{\theta_2^j} f_1(\sigma_1(\theta_1),\widetilde{\sigma}_2^*(\theta_2^j),\theta_1^i,\theta_2^j)p(\theta_1,\theta_2)d\theta_1d\theta_2\\
\notag&+ML_p(\overline{\theta}_1-\underline{\theta}_1)(\overline{\theta}_2-\underline{\theta}_2)\epsilon_0\\
\notag\leq&\sum_{i=1}^{N_1}\sum_{j=1}^{N_2}\int_{\theta_1^{i-1}}^{\theta_1^{i}}\int_{\theta_2^{j-1}}^{\theta_2^j}f_1(\sigma_1(\theta_1),\widetilde{\sigma}_2^*(\theta_2^j),\theta_1,\theta_2)p(\theta_1,\theta_2)d\theta_1d\theta_2+C_0\epsilon_0\\
\notag=&\int_{\underline{\theta}_1}^{\overline{\theta}_1}\int_{\underline{\theta}_2}^{\overline{\theta}_2}f_1(\sigma_1(\theta_1),\widetilde{\sigma}_2^*(\theta_2),\theta_1,\theta_2)p(\theta_1,\theta_2)d\theta_1d\theta_2+C_0\epsilon_0,
\end{align}
where $C_0=2L_\theta+ML_p(\overline{\theta}_1-\underline{\theta}_1)(\overline{\theta}_2-\underline{\theta}_2)$. Then applying \eqref{eq_pfap} to \eqref{eq_transinprf}, for any $\sigma_1\in\Sigma_1$,  
\begin{equation*}
\begin{aligned}
&\int_{\underline{\theta}_1}^{\overline{\theta}_1}\int_{\underline{\theta}_2}^{\overline{\theta}_2}f_1(\widetilde{\sigma}_1^*(\theta_1),\widetilde{\sigma}_2^*(\theta_2),\theta_1,\theta_2)p(\theta_1,\theta_2)d\theta_1d\theta_2\\
\leq&\int_{\underline{\theta}_1}^{\overline{\theta}_1}\int_{\underline{\theta}_2}^{\overline{\theta}_2}f_1(\sigma_1(\theta_1),\widetilde{\sigma}_2^*(\theta_2),\theta_1,\theta_2)p(\theta_1,\theta_2)d\theta_1d\theta_2+C_1\epsilon_0,
\end{aligned}%
\end{equation*}
namely $EU(\widetilde{\sigma}^*_1,\widetilde{\sigma}_2^*)\leq EU(\sigma_1,\widetilde{\sigma}_2^*)+C_1\epsilon_0$, where $C_1=C_0+2L_\theta$. Similarly, for any $\sigma_2\in\Sigma_2$, $
EU(\widetilde{\sigma}^*_1,\widetilde{\sigma}_2^*)\geq EU(\widetilde{\sigma}^*_1,\sigma_2)+C_1\epsilon_0$. Hence, a DBNE is an $\epsilon_1$-BNE with $\epsilon_1=C_1\epsilon_0$. 

Next, we prove that the DBNE converges to the BNE. By the zero-sum condition, 
\begin{equation*}
EU(\sigma_1^*,\sigma_2^*)\geq EU(\sigma_1^*,\widetilde{\sigma}_2^*)\geq EU(\widetilde{\sigma}_1^*,\widetilde{\sigma}_2^*)-\epsilon_1\geq EU(\widetilde{\sigma}_1^*,\sigma_2^*)-2\epsilon_1\geq EU(\sigma_1^*,\sigma_2^*)-2\epsilon_1.%
\end{equation*}
Thus, $EU(\widetilde{\sigma}_1^*,\sigma_2^*)-EU(\sigma_1^*,\sigma_2^*)\leq 2\epsilon_1$. Since $f_l$ is $\mu$-strongly convex, for any $\theta_1\in\Theta_1$ and $\theta_2\in\Theta_2$, 
\begin{equation*}
\begin{aligned}
&f_1(\widetilde{\sigma}_1^*(\theta_1),\sigma^*_2(\theta_2),\theta_1,\theta_2)- f_1(\sigma_1^*(\theta_1),\sigma_2^*(\theta_2),\theta_1,\theta_2)\\\geq&(\nabla_1 f_1(\sigma_1^*(\theta_1),\sigma_2^*(\theta_2),\theta_1,\theta_2))^T(\widetilde{\sigma}_1^*(\theta_1)-\sigma_1^*(\theta_1))+\frac{\mu}{2}\lVert \widetilde{\sigma}_1^*(\theta_1)-\sigma_1^*(\theta_1)\rVert^2.
\end{aligned} 
\end{equation*}
Because for any $\theta_1\in\Theta_1$, $(\nabla_1U_1(\sigma_1^*(\theta_1),\sigma_2^*(\theta_2),\theta_1))^T( \widetilde{\sigma}_1^*(\theta_1)-\sigma_1^*(\theta_1))\geq0$,  
\begin{equation*}
2\epsilon_1\geq EU(\widetilde{\sigma}_1^*,\sigma_2^*)-EU(\sigma_1^*,\sigma_2^*)\geq \frac{\mu}{2}\int_{\underline{\theta}_1}^{\overline{\theta}_1}\int_{\underline{\theta}_2}^{\overline{\theta}_2}\lVert \widetilde{\sigma}_1^*(\theta_1)-\sigma_1^*(\theta_1)\rVert^2 p(\theta_1,\theta_2)d\theta_1d\theta_2=\frac{\mu}{2}\lVert \widetilde{\sigma}_1^*-\sigma_1^*\rVert^2_{\mathcal{H}_1},
\end{equation*}
namely $\lVert \widetilde{\sigma}_1^*-\sigma_1^*\rVert^2_{\mathcal{H}_1}\leq\frac{4\epsilon_1}{\mu}$. Similarly, $\lVert \widetilde{\sigma}_2^*-\sigma_2^*\rVert^2_{\mathcal{H}_2}\leq\frac{4\epsilon_1}{\mu}$. As $N_1$ and $N_2$ tend to infinity, $\epsilon_1$ tends to 0, and thus the DBNE $(\widetilde{\sigma}_1^*,\widetilde{\sigma}_2^*)$ converges to the BNE $(\sigma_1^*,\sigma_2^*)$.%

\section{Proof of Proposition \ref{prop_compress_graph}}\label{pf_com_graph}
According to Assumption \ref{ass_game}(vi), each agent $v_l^i$ is able to communicate with agents in $\Xi_l$ once every $\mathcal{R}_0$ iterations and have access to messages from $\Xi_{3-l}$ every $\mathcal{S}_0$ iterations, which means that $v_l^i$ can receive nonempty messages $[x(t)]_k$ ($x(t)\in\mathbb{R}^{N_{l}m_{l}}$) at each entry $k$ from $\Xi_l$ in $\mathcal{R}_0\lceil N_lm_l/d_l\rceil$ iterations and nonempty messages $[y(t)]_k$ ($y(t)\in\mathbb{R}^{N_{3-l}m_{3-l}}$) at each entry $k$ from $\Xi_{3-l}$ in $\mathcal{S}_0\lceil N_{3-l}m_{3-l}/d_{3-l}\rceil$ iterations. Thus, the generated graph sequences $\mathcal{G}_{l,k}(t)$ satisfy Proposition \ref{prop_compress_graph} with $\mathcal{R}=\max_l\{\mathcal{R}_0\lceil N_lm_l/d_l\rceil\}$ and $\mathcal{S}=\mathcal{S}_0\mathcal{R}$. 

\section{Proof of Lemma \ref{prop_compress}}\label{pf_com}
Due to Jensen's inequality, 
\begin{equation*}\begin{aligned}
\lVert\hat{\sigma}_{l,i}(t)-\bar{\sigma}_l(t)\rVert\leq &\sum_{k=1}^{N_lm_l} \lVert[\hat{\sigma}_{l,i}(t)]_k-[\bar{\sigma}_l(t)]_k\rVert
\\\leq&\sum_{k=1}^{N_lm_l}\sum_{j=1}^{n_l} [A_{l,k}(t)]_{ij}\lVert[\sigma_{l,j}(t)]_k-[\bar{\sigma}_l(t)]_k\rVert
\leq\sqrt{N_lm_l} \varepsilon_0.
\end{aligned} 
\end{equation*}
At time $t$, denote the last time agent $v_l^i$ receive nonempty messages from its rivals at the entry $k$ ($k\in\{1,\dots,N_lm_l\}$) by $s_k(t)$. From Proposition \ref{prop_compress_graph}(b), $u_k(t)\leq \mathcal{S}$. For any $t>0$, 
\begin{equation*}
\begin{aligned}
&\lVert\hat{\zeta}_{3-l,i}(t)-\bar{\sigma}_{l}(t)\rVert\leq \sum_{k=1}^{N_lm_l}\lVert [\hat{\zeta}_{3-l,i}(s_k(t))]_k-[\bar{\sigma}_l(t)]_k\rVert\\
\leq& \sum_{k=1}^{N_lm_l}\big(\lVert[\bar{\sigma}_l(s_k(t))]_k-[\bar{\sigma}_l(t)]_k\rVert+\lVert [\hat{\zeta}_{3-l,i}(s_k(t))]_k-[\bar{\sigma}_l(s_k(t))]_k\rVert\big)\\
\leq& \sum_{k=1}^{N_lm_l}\bigg(\sum_{j=s_k(t)}^{t} \lVert [\bar{\sigma}_l(j)]_k-[\bar{\sigma}_l(j+1)]_k\rVert\\&+\sum_{j=1}^{n_l}[C_{l,k}(s_k(t))]_{ij}\lVert[\hat{\sigma}_{l,j}(s_k(t))]_k-[\bar{\sigma}_l(s_k(t))]_k\rVert\bigg)\leq N_lm_l\varepsilon_0+\mathcal{S}\sqrt{N_lm_l}\varepsilon_1.
\end{aligned} 
\end{equation*}

\section{Proof of Proposition \ref{prop_penalty}}\label{pf_penalty}
For any $x,y\in\mathbb{R}^{m_l}$ and $\lambda\in[0,1]$, denote the projection by $x'=\Pi_{\mathcal{X}_l}(x)$ and $y'=\Pi_{\mathcal{X}_l}(x)$. Since $\mathcal{X}_l$ is convex, $\lambda x'+(1-\lambda)y'\in\mathcal{X}_l$ for $0<\lambda<1$. Therefore, $\lVert x-x'\rVert\leq \lVert x-z\rVert$ for any $z\in\mathcal{X}_l$. 
\begin{equation*}
\begin{aligned}
H_l(\lambda x+(1-\lambda) y)=&E_l\lVert\lambda x+(1-\lambda) y-\Pi_{\mathcal{X}_l}(\lambda x+(1-\lambda) y)\rVert\\
\leq&E_l\lVert \lambda x+(1-\lambda) y-(\lambda x'+(1-\lambda) y')\rVert\\
\leq&\lambda E_l\lVert x-x'\rVert+(1-\lambda)E_l\lVert y-y'\rVert=\lambda H_l(x)+(1-\lambda)H_l(y).
\end{aligned} \end{equation*}
Thus, $H_l$ is convex. Moreover,
\begin{equation*}
\begin{aligned}
E_l\lVert x-y\rVert=E_l\lVert x-x'+x'-y\rVert\geq H_l(x)-E_l\lVert x'-y\rVert
\geq H_l(x)-E_l\lVert y'-y\rVert&\\=H_l(x)-H_l(y).&
\end{aligned}\end{equation*}
Similarly, $E_l\lVert x-y\rVert\geq H_l(y)-H_l(x)$. Therefore, $H_l$ is $E_l$-Lipschitz continuous in $x\in\mathbb{R}^{m_l}$.

By the properties of projection, $\left<z-x',x-x'\right>\leq 0$ for any $z\in\mathcal{X}_l$. Then

\begin{equation*}
\begin{aligned}
&\left<y-x,h_l(x)\right>=\left<y-y'+y'-x'+x'-x,E_l\tfrac{x-x'}{\lVert x-x'\rVert}\right>\\
\leq&\left<y-y',E_l\tfrac{x-x'}{\lVert x-x'\rVert}\right>-H_l(x)\leq E_l\lVert y-y'\rVert \tfrac{\lVert x-x'\rVert}{\lVert x-x'\rVert}-H_l(x)=H_l(y)-H_l(x).
\end{aligned}\end{equation*}
Thus, $h_l$ is a subgradient of $H_l$.

Fix $\sigma_2\in\mathbb{R}^{N_lm_l}$. Denote the optimal strategy of $\Xi_1$ by $\sigma_1'\in\mathcal{X}_l$ when $\Xi_2$ adopts $\sigma_2$, which satisfies $\sigma_1^{r '}=\arg\min_{\sigma_1^r\in\mathcal{X}_l}\widetilde{U}_1(\sigma_1,\sigma_2,\theta_1^r)$ for $r\in\{1,\dots,N_1\}$. For any $\sigma^r_1\in\mathbb{R}^{m_l}$,

\begin{equation*}
\begin{aligned}
&\widetilde{U}_1(\sigma_1,\sigma_2,\theta_1^r)+H_1(\sigma_1^r)-\widetilde{U}_1(\sigma'_1,\sigma_2,\theta_1^r)\\\geq &\widetilde{U}_1(\sigma_1,\sigma_2,\theta_1^r)+H_1(\sigma_1^r)-\widetilde{U}_1(\Pi_{\widetilde{\Sigma}}(\sigma_1),\sigma_2,\theta_1^r)\\
\geq&-L_{1,1}\lVert \sigma_1^r-\Pi_{\mathcal{X}_1}(\sigma_1^r)\rVert+E_1\lVert \sigma_1^r-\Pi_{\mathcal{X}_1}(\sigma_1^r)\rVert\geq 0.
\end{aligned}\end{equation*}
Then we obtain a similar result for $\Xi_2$. Therefore, for $l\in\{1,2\}$, the minimizer of $\widetilde{U}_l+H_l$ in $\mathbb{R}^{m_l}$ lies in $\mathcal{X}_l$.


\end{document}